\documentclass{amsart}
\usepackage[colorlinks, linkcolor=blue,  citecolor=blue, urlcolor=blue]{hyperref}%
\usepackage{amsthm}
\usepackage{url}
\usepackage{xspace}
\usepackage{graphicx}
\usepackage{subfig} 
\usepackage{enumerate}
\usepackage{verbatim}
\usepackage{alltt}
\usepackage{color}
\usepackage[alphabetic]{amsrefs}
\graphicspath{{figures/}}
\usepackage{siunitx}

\newtheorem{theorem}{Theorem}
\newtheorem{proposition}{Proposition}
\theoremstyle{remark}
\newtheorem{remark}{Remark}
\newtheorem{definition}{Definition}

\newtheorem{example}{Example}
\newtheorem{lemma}{Lemma}
\numberwithin{equation}{section}

\newcommand{\bq}{\begin{equation}}
\newcommand{\eq}{\end{equation}}
\newcommand{\R}{\mathbb{R}}
\newcommand{\Rd}{\R^d}

\newcommand{\bO}{\mathcal{O}}

\newcommand{\USC}{\text{USC}}
\newcommand{\LSC}{\text{LSC}}

\newcommand{\usub}{\overline{u}}
\newcommand{\usup}{\underline{u}}

\begin{document}

\title[Filtered Schemes for Hamilton-Jacobi]
{Filtered schemes for {H}amilton-{J}acobi equations: a simple construction of convergent accurate difference schemes}

\author{Adam M. Oberman}
\thanks{Department of Mathematics and Statistics, McGill University, 805 Sherbrooke Street West, Montreal, Quebec, H3A 0G4, Canada ({\tt adam.oberman@mcgill.ca})}

\author{Tiago Salvador}
\thanks{Department of Mathematics and Statistics, McGill University, 805 Sherbrooke Street West, Montreal, Quebec, H3A 0G4, Canada ({\tt tiago.saldanhasalvador@mail.mcgill.ca})
}
\thanks{partially supported by FCT doctoral grant SFRH / BD / 84041 /2012
}

\begin{abstract}
We build a simple and general class of finite difference schemes for first order Hamilton-Jacobi (HJ) Partial Differential Equations. These filtered schemes are convergent to the unique viscosity solution of the equation.  The schemes are accurate: we implement second, third and fourth order accurate schemes in one dimension and second order accurate schemes in two dimensions, indicating how to build higher order ones. They are also explicit, which means they can be solved using the fast sweeping method.
The accuracy of the method is validated with computational results for the eikonal equation and other HJ equations in one and two dimensions, using filtered schemes made from standard centered differences, higher order upwinding and ENO interpolation.
\end{abstract}

\date{\today}




\maketitle

\section{Introduction}\label{sec:intro}

In this work we build a simple and general class of finite difference schemes for first order Hamilton-Jacobi (HJ) Partial Differential Equations.  These filtered schemes are almost monotone (in a rigorous sense) and thus provably convergent to the unique viscosity solution of the equation. 
The schemes are formally accurate: we implement second, third and fourth order accurate schemes in one dimension and second order accurate schemes in two dimensions, indicating how to build higher order ones. They are also explicit, which means they can be solved using the fast sweeping method~\cites{FastSweeping,FastSweepingZhao}, or the fast marching method~\cites{FastMarching,TsitsiklisFastMarching} in the case of the eikonal equation.

There are already a large number of discretizations and solvers available for  Hamilton-Jacobi equations.   Our filtered schemes are designed to remain stable while allowing for a wide choice of accurate discretizations.
The simplest approximations are finite difference schemes based on a Cartesian grid.  In this class, monotone schemes are provably convergent \cite{BSnum}, but only first order accurate~\cite{ObermanDiffSchemes}.  
In general, higher order finite difference schemes for HJ equations are neither monotone, nor stable.  For example, the centered difference scheme is unstable for the eikonal equation \cite[Section 4.3]{SethianBook}. 

Higher order accurate schemes have been built, but only by giving up other desirable properties (e.g.\ ease of implementation, fast solvers, or the convergence proof). Semi-Lagrangian schemes \cites{falcone2002semi,cristiani2007fast}, are accurate, but they involve solving the characteristic ordinary differential equations, and are generally more complicated to implement. Central schemes \cite{LinTadmorHJ}  achieve second order accuracy, at the expense of a slightly more complicated, non-explicit formulation. The ENO and WENO schemes \cites{OsherShuENO,ShuHighOrder,JiangWENO} are accurate, and while not provably convergent, they are effective in practice. Combinations of WENO and central schemes have been implemented, achieving higher order accuracy~\cite{bryson2003high}. The ENO based schemes use adaptive  stencils, which complicates the use of fast solvers (however see \cite{ZhangHighOrderFastSweeping} for a sweeping method). Fast marching methods require specialized data structures to implement, are usually first order accurate (however see \cites{Ahmed} for higher order methods) and only apply to the eikonal equation.  A compact upwind second order scheme for the eikonal equation was  proposed in \cite{BenamouSongtingZhao}.   

A higher order scheme for Hamilton-Jacobi equations was presented by Abgrall in~\cite{Abgrall}.  Since this scheme uses some ideas similar to ours, we discuss it in further detail in the next paragraphs.

\subsection{Contribution of this work}\label{sec:contribution}  

We build filtered schemes by combining a stable, monotone scheme with an accurate (but possibly unstable) scheme. The accurate scheme is not required to be stable on its own: it can simply be standard higher order finite differences, or it can designed to take advantage of known properties of the solutions to the equation under consideration (for example a compact scheme could better avoid singularities in the solution). However, independently of the choice made, the combination of the two schemes is both provably convergent, and (potentially) higher order accurate.  We demonstrate that with a judicious choice of the accurate scheme, the higher order accuracy can be achieved.   In particular, using one-sided higher order finite differences for the accurate scheme, combined with an upwind monotone scheme results in a very simple, explicit, and accurate scheme for the eikonal equation.   We also treat more general cases.

The proof of convergence relies on the classical and well known Barles-Souganidis result \cite{BSnum}, which states that monotone, stable, consistent schemes converge.  In this paper, convergence of ``almost monotone'' schemes was mentioned as a remark, but no definition or examples were given.   It turns out that filtered schemes, the way we define them, fit very naturally into the framework of the proof, while also being general enough to allow for a variety of schemes.  
The recent (2009) paper by Abgrall~\cite{Abgrall} was the first paper to present a provably convergent scheme that blends a monotone scheme with an accurate scheme. The convergence of this scheme 
also follows from an adaptation of the Barles-Souganidis convergence proof.
The small (uniformly bounded) correction to the scheme due to the lack of monotonicity can be absorbed into the term usually seen as the consistency error. The idea of a filtered scheme is then to provide a systematic method to blend a monotone scheme with an accurate scheme and thereby allowing for higher order accuracy. Filtered schemes were previously introduced in  \cite{FroeseObermanFiltered} in the context of  the Monge-Amp\`ere equation.
There they were used to overcome the reduction in accuracy based on the wide-stencil monotone scheme.  However, the filtered schemes can be applied in a different context to  build higher order accurate schemes for the eikonal equation and for more general Hamilton-Jacobi equations.
 
The schemes we introduce have the following properties
\begin{enumerate}
\item They are simple and easy to implement on Cartesian grids. For example, for the eikonal equation the filtered scheme using the centered difference scheme, is convergent and second order accurate, which results in the simplest second order accurate difference scheme.
\item Higher order explicit schemes are obtained using higher order upwind interpolation. These higher order schemes can be solved using fast sweeping.  If desired, fast marching can be used  instead in the case of the eikonal equation.
\item Other choices of accurate schemes can be used instead: we implement ENO schemes for comparison.  Any choice of discretization (e.g. the popular discontinuous Galerkin method) can be used, provided  a monotone scheme can also be constructed in the same setting.
\item For the eikonal equation in one dimension, higher order convergence rates for the numerical  solution is proved, even for non-smooth solutions.
\item For HJ equations (in general), higher order convergence is obtained locally, in regions where the solution is smooth.
\end{enumerate}

\subsection{The eikonal equation}
We take a particular interest on the eikonal equation
\bq\begin{cases}\label{eikonal}
|\nabla u(x)| = f(x),	& \text{for $x$ outside }\Gamma,\\
u(x) = g(x),			& \text{for $x$ on }\Gamma.
\end{cases}
\eq
where $f > 0$ and $\Gamma$ is here a closed, bounded set. The eikonal equation has wide applications in geometric optics, computer vision, optimal control, etc. Moreover, as pointed out in \cite{BenamouSongtingZhao}, high order schemes are particularly important in the high frequency wave propagation where the eikonal equation is coupled to a transport equation through its gradient \cites{Qian99anadaptive,Symes94kirchhoffsimulation}.

\subsection{Hamilton-Jacobi equations}
We consider HJ equations of the form
\bq\label{HJequation}\begin{cases}
H(x,\nabla u) = f(x),	& x \in \Omega,\\
u(x) = g(x),			& x \in \Gamma,
\end{cases}
\eq
where $\nabla u$ is the gradient of the function $u$, $\Omega$ is an open set, $\Gamma$ is the boundary of $\Omega$ and the Hamiltonian $H$ is a nonlinear Lipschitz continuous function. HJ equations appear in many applications, such as optimal control, differential games, image processing, computer vision and geometric optics. We always refer to the eikonal equation specifically, even though it's in fact an HJ equation (take $H(p) = |p|$). When we refer to HJ equations we always have more general equations in mind.

In general, solutions are not smooth (or even differentiable) and so we consider viscosity solutions  (see Appendix \ref{appx:viscosity}, \cite{CIL} in general, \cite{Abgrall} in this context). The viscosity solutions can be piecewise smooth with a singularity in the gradient. It therefore makes sense to design high order schemes that provide higher order accuracy (at least) away from these singularities.

\subsection{The definition of the filtered scheme}
The filtered schemes are defined by the following.
Let $F^h_M$ denote the monotone discretization of the operator on the grid with spacing $h$, given  below in subsection~\ref{sec:monotone}. Let $F^h_A$ denote an accurate discretization of the same operator, with several possible choices being given below in subsection~\ref{sec:schemes}.

Then the filtered scheme, $F^h$, blends the two schemes together by using the following simple formula:
\bq\label{defnFiltered}
F^h[u] = \begin{cases}
F^h_A[u], & \text{ if } \left|F^h_A[u]-F^h_M[u]\right| \leq \sqrt{h}\\
F^h_M[u],   & \text{otherwise.}
\end{cases}
\eq
The filtered scheme, which is consistent provided both underlying schemes are consistent, is usually  not monotone.  However it is almost monotone, since, by definition, 
\bq\label{consistencyforfiltered}
F^h[u] = F^h_M[u] + \bO(h^{1/2}).
\eq
The proof in \cite{BSnum} can then be modified to include these schemes since the term of $\bO(h^{1/2})$ can be absorbed into the truncation error.

\begin{remark}\label{remark:rates}
The choice of the factor $\sqrt{h}$ in \eqref{defnFiltered} is designed to fit between two rates: large enough to permit the accurate scheme to be active where the solution is smooth, and small enough to force the monotone scheme to be active when the solution is singular.  So, for example, for the eikonal equation, the monotone scheme is accurate to $\bO(h)$, and the accurate scheme is $\bO(h^2)$ or better, so we can take the factor $\sqrt{h}$.  (If we took it to be $h$, we might fail to see the monotone scheme, and get something less stable).

Below, at the end of subsection \ref{sec:examples1D},  we consider an example where the Hamiltonian is non-convex, and the observed convergence rate is $\bO(\sqrt{h})$, for the monotone scheme, and so we take the factor to be smaller than $\sqrt{h}$.
\end{remark}
The following convergence theorem, in a more general setting, was proved in~\cite{FroeseObermanFiltered}.  For the convenience of the reader we include the proof, specialized to our case, in the Appendix \ref{appendix}.

\begin{theorem}[Convergence of Approximation Schemes]\label{thm:converge}
Let $u$ be the unique viscosity solution of \eqref{HJequation}.
For each $h>0$,  let $u^h$ be a stable solution of $F^h[u] = 0$, where the filtered scheme $F^h$ is given by ~\eqref{defnFiltered} and $F^h_M$ is consistent and monotone. Then 
\[
u^h \to u, \quad \text{ locally uniformly,  as } h \to 0.
\]
\end{theorem}

To apply the theorem, we do not need to know that solutions of the filtered scheme are unique.  However, we do need to know that stable solutions exist.  Existence of such solutions was proven in \cite{FroeseObermanFiltered} for a slightly different form of the filtered scheme.  Instead of \eqref{defnFiltered}, a continuous interpolation between the monotone and accurate scheme was used.   This was required for the continuity argument in the proof of existence, and it was also of practical use for a Newton solver.  

In our setting, although discontinuous, \eqref{defnFiltered} has a simpler form, which allows for explicit solution formulas below.  These explicit solution formulas allow us to build fast sweeping solvers, which are appropriate for Hamilton-Jacobi equations.  In practice the computational results are as good as could be expected.  For the purpose of the proof, a continuous filter is needed but the practical advantages of the discontinuous one outweigh the lack of rigor.

Theorem \ref{thm:converge} does not provide any information regarding the convergence rate. Proving higher order convergence requires additional efforts and is possible in specific settings. 
For the one-dimensional eikonal equation, we prove higher order convergence in subsection~\ref{sec:error}. For the two-dimensional eikonal equation, second and third order convergence is proven for smooth solutions in \cite{Ahmed}.  We are more interested in demonstrating the higher order convergence in practice, which is done using numerical simulations. In particular, in the case of piecewise smooth solutions in two dimensions, we achieve second order convergence rates in the smooth region, and first order convergence overall in the $l^\infty$ norm.

\begin{remark}  
In addition to stationary equations, we can build filtered schemes for time dependent equations.  This can be accomplished by using the filtered scheme on the spatial part of the operator, and a standard time discretization (forward Euler or strong stability preserving time discretizations) for the time derivative. As needed, the filter could also be applied to the time derivative term as well.
In this case, with minor modifications, the proof of convergence for the filtered scheme goes through, since, as it's standard for viscosity solution, the time derivative can be considered as an additional spatial variable.
\end{remark}


\section{Discretization and solvers}
\label{sec:compute}

In this section we will discuss the discretization of the monotone and  filtered schemes for both HJ and eikonal equations for different choices of the accurate schemes (centered, upwind and ENO). We do this both in one and two dimensions. We recall that our filtered schemes are given by \eqref{defnFiltered}. We should point out that all discretizations for HJ equations can be applied to the eikonal equation, although we choose to present and use specific discretizations for the eikonal equation given its importance in the literature.

We consider only the case of regular Cartesian grids since the discretization is simpler and the idea is clear. It is certainly possible to build filtered schemes using higher order methods on triangulated grids for example.

\subsection{Monotone schemes}\label{sec:monotone}


For the eikonal equation, in the one-dimensional case, the monotone scheme is given by
\bq\label{monotone1deikonal}
|u_x^h|^M = \max\left\{-\frac{u(x+h)-u(x)}{h},\frac{u(x)-u(x-h)}{h},0\right\}.
\eq
Since we are working on a Cartesian grid, extending it to the two dimensional case simply requires the use of the standard Euclidean 2-norm function $N:\R^2 \to \R$ given by
\bq\label{norm}
N(x,y) = \sqrt{x^2+y^2}.
\eq
We then define
\bq\label{monotone2deikonal}
|\nabla u^h|^M = N\left(|u_x^h|^M,|u_y^h|^M\right)
\eq
which is monotone,  as desired (see for example~\cite{ObermanDiffSchemes}).

There are several monotone numerical Hamiltonians we could use to discretize HJ equations. 
Here we choose to use  the Lax-Friedrichs numerical Hamiltonian \cite{LFHJ}, because it has a simple form and it can be used for both convex and nonconvex Hamiltonians:
\bq\label{monotone1dHJ}
H^h_{LF}[u](x) = H^h_{LF}(x,p^+,p^-) = H\left(x,\frac{p^++p^-}{2}\right)-\frac{1}{2}\sigma_x(p^+-p^-)
\eq
where $\sigma_x$ is the artificial viscosity satisfying $\sigma_x = \max \left|\frac{\partial H}{\partial p}\right|$, $p = u_x$ and $p^{\pm}$ are the corresponding forward and backward differences approximations of $u_x$.

The scheme easily generalizes into higher dimensions: in the two-dimensional case we have
\begin{align}\label{monotone2dHJ}
\begin{aligned}
H^h_{LF}[u](x,y)	& = H^h_{LF}(x,y,p^+,p^-,q^+,q^-)\\
					& = H\left(x,y,\frac{p^++p^-}{2},\frac{q^++q^-}{2}\right)-\sigma_x\frac{p^+-p^-}{2}-\sigma_y\frac{q^+-q^-}{2}
\end{aligned}
\end{align}
where $\sigma_y = \max \left|\frac{\partial H}{\partial q}\right|$, $q = u_y$ and $q^{\pm}$ are the corresponding forward and backward differences approximations of $u_y$.

\subsection{Accurate schemes}\label{sec:schemes}
We know that the filtered scheme will converge independently of the choice of the accurate scheme. Its purpose is to provide additional accuracy in the regions where the solution is smooth and where the accurate scheme is active. Thus the resulting  accuracy of the solution comes from a judicious choice of the accurate scheme. In addition to the accuracy, the choice of accurate scheme determines the type of solver we can use (iterative or sweeping), based on whether an explicit solution formula is available (see subsection \ref{sec:explicit}).

We first consider the one-dimensional case and then show how, as in the previous section, the schemes can be generalized for the two-dimensional case. 

\textbf{Centered Schemes:} The second order accurate centered scheme are obtained by simply replacing $u_x$ by its second order centered approximation:
\begin{align*}
& |u_x^h|^{C,2} = \frac{|u(x+h)-u(x-h)|}{2h},\\
& H^h_{C,2}[u](x) = H\left(x,\frac{u(x+h)-u(x-h)}{2h}\right).
\end{align*}

\textbf{Upwind Schemes:} The upwind schemes proposed here were first thought for the eikonal equation, although they can be generalized to HJ equations in general. In the eikonal equation case, they are designed to choose the finite difference stencil in terms of the direction of the characteristics of the solution. This means using the left (right) biased stencil if the characteristics are being propagated from the left (right). The higher order upwind schemes generalize the monotone scheme above. They are defined as follows.

Set $P^{\pm,n}[u]$ to be the interpolating polynomial of degree $n$ of $u$ at the nodes $x_j = x \pm jh$ for $j = 0,1,\ldots,n$. (The sign in the superscript indicates interpolation to the left or to the right.) These interpolating polynomials are standard and given in several convenient explicit forms (see \cite{Iserlesbook}). We give a specific example below. We then set
\begin{align*}
& |u_x^h|^{U,n} = \max\left\{\frac{d}{dx}P^{+,n}[u](x),-\frac{d}{dx}P^{-,n}[u](x)\right\}\\
& H^h_{U,n}[u](x) = H^h_{LF}\left(x,\frac{d}{dx}P^{+,n}[u](x),\frac{d}{dx}P^{-,n}[u](x)\right)
\end{align*}

\textbf{ENO Schemes:} High order essentially non-oscillatory (ENO) are another option for the accurate discretization. (A refinement of ENO is WENO \cite{JiangWENO}, which we choose not to implement, since the main idea is clear from the ENO examples.) The  idea underlying the ENO schemes is to do a standard interpolation using an adaptive stencil, i.e., the stencil used depends on the function being interpolated. Starting with two nodes, the ENO interpolation of order $n$ selects the remaining $n-1$ interpolation nodes by successively adding nodes to the stencil with the smallest Newton divided difference. This way,  the $r^{th}$ node is chosen by comparing two approximations of the derivative of order $r+1$, with $r$ taking successively the values $\{1,\ldots,n-1\}$. 
 
Let $E^{n,\pm\frac{1}{2}}[u]$ denote the ENO interpolation as explained above, and as defined in \cite{OsherShuENO}. Then we define the $n^{th}$-order accurate ENO scheme to be
\begin{align*}
& |u_x^h|^{E,n} = \max\left\{\frac{d}{dx}E^{n,\frac{1}{2}}[u](x),-\frac{d}{dx}E^{n,-\frac{1}{2}}[u](x)\right\}\\
& H^h_{E,n}[u](x) = H^h_{LF}\left(x,\frac{d}{dx}E^{n,\frac{1}{2}}[u](x),\frac{d}{dx}E^{n,-\frac{1}{2}}[u](x)\right)
\end{align*}

\textbf{Two dimensional schemes.} 
In the case of the eikonal equation we use \eqref{norm} as we did in subsection \ref{sec:monotone}. The second order centered scheme becomes
\bq
|\nabla u^h|^{C,2} = N\left(|u_x^h|^{C,2},|u_y^h|^{C,2}\right),
\eq
the upwind schemes become
\bq\label{upwind2d}
|\nabla u^h|^{U,n} = N\left(|u_x^h|^{U,n},|u_y^h|^{U,n}\right),
\eq
and, finally, the ENO schemes are defined as
\bq
|\nabla u^h|^{E,n} = N\left(|u_x^h|^{E,n},|u_y^h|^{E,n}\right).
\eq
The upwind schemes here defined for the eikonal equation recover the  $2^{nd}$ and $3^{rd}$ order upwind schemes from (\cite{FastMarching}, \cite{Chopp} and \cite{Ahmed}). These schemes have been solved using Fast Marching algorithms.

As for HJ equations, the extension to two dimensions follows from using the two-dimensional expression of $H^h_{LF}$ as we did with the monotone scheme.

\subsection{Explicit methods}\label{sec:explicit}
For upwind schemes, the interpolation is fixed, so we can solve for the reference variable and build explicit schemes. In contrast, it is difficult to directly build explicit methods for many of the other schemes.
Rather than present the general method for solving for the reference variable and in order to be concrete and save space, we give a specific example below. The general method should then be clear.

\subsubsection*{Eikonal equations}
\begin{example}[one-dimensional case]
Consider first the monotone scheme in the one-dimensional case \eqref{monotone1deikonal}. Solving the equation $|u_x^h|^M = f$ for the reference variable, $u(x)$, leads to
\bq\label{Emon}
u(x) = \min\{u(x+h),u(x-h)\}+hf(x).
\eq
Consider now the second order upwind scheme, again in one dimension. The upwind scheme takes the form 
\[
|\nabla u_x^h|^{U,2} \equiv \frac{1}{2h}\max\left\{3u(x)-4u(x\pm h)+u(x \pm 2h)\right\} = f.
\]
Solving the preceding equation  for the reference variable, $u(x)$,  leads to
\bq\label{Eup}
u(x) = \frac{1}{3}\min\{4u(x+h)-u(x+2h),4u(x-h)-u(x-2h)\}+\frac{2}{3}hf(x).
\eq

Finally, consider the correspondent filtered scheme. Combining~\eqref{Emon} and~\eqref{Eup} and using the definition of the filtered scheme~\eqref{defnFiltered} we obtain the following explicit representation of the solution of the filtered scheme at a reference point in terms of the neighboring values
\[
u(x) = \begin{cases}
\frac{1}{3}\min\{4u(x\pm h)-u(x\pm 2h)\}+\frac{2}{3}hf(x) & \text{if } \left||u_x^h|^A-|u_x^h|^M\right| \leq \sqrt{h},\\
\min\{u(x+h),u(x-h)\}+hf(x) & \text{otherwise.}
\end{cases}
\]
\end{example}

\begin{example}[two-dimensional case]
We can also obtain an explicit solution for the filtered schemes using the upwind scheme in the two-dimensional case as above. In this case solving for the reference variable $u(x,y)$ requires  solving a nonlinear equation  of the form 
\[
\left[(z-a)^+\right]^2+\left[(z-b)^+\right]^2 = c^2
\]
for the unknown $z$ where $a$, $b$ and $c > 0$ are constants and $(z)^+:=\max\{z,0\}$. This equation combines  piecewise linear functions with a quadratic function. The unique solution of the equation is given by
\bq\label{explicitgeneric2d}
z = \begin{cases}
\min\{a,b\}+c						& |a-b| \geq c,\\
\frac{a+b+\sqrt{2c^2-(a-b)^2}}{2}	& |a-b| < c,
\end{cases}
\eq
(see e.g.\ \cite{FastSweepingZhao} for a  derivation).

In the case of the monotone scheme we get
\[\begin{cases}
a = \min\{u(x+h,y),u(x-h,y)\},\\
b = \min\{u(x,y+h),u(x,y-h)\},\\
c = h f(x).
\end{cases}\]
As for the second order upwind scheme we have
\[\begin{cases}
a = \frac{1}{3}\min\{4u(x\pm h,y)-u(x\pm 2h,y)\},\\
b = \frac{1}{3}\min\{4u(x,y\pm h)-u(x,y\pm 2h)\},\\
c = \frac{2}{3}h f(x).
\end{cases}\]

The explicit formula of the filtered scheme can then be obtained as in the one-dimensional case using the definition of filtered scheme \eqref{defnFiltered} and \eqref{explicitgeneric2d}.
\end{example}

\subsubsection*{Hamilton-Jacobi equations}
\begin{example}[one-dimensional case]
Consider first the monotone scheme \eqref{monotone1dHJ}. We know that
\[p^+ = \frac{u(x+h)-u(x-h)}{h}, \quad p^- = \frac{u(x)-u(x-h)}{h}.\]
Thus, solving $H^h_{LF}[u] = f$ for the reference variable, $u(x)$, leads to
\[u(x) = \frac{1}{\sigma_x}\left[f(x)-H\left(x,\frac{u(x+h)-u(x-h)}{2h}\right)+\sigma_x\frac{u(x+h)+u(x-h)}{2h}\right].\]

Consider now the second order upwind scheme. We have
\begin{align*}
\frac{d}{dx}P^{+,2}[u](x)	& = \frac{-3u(x)+4u(x+h)-u(x+2h)}{2h},\\
\frac{d}{dx}P^{-,2}[u](x)	& = \frac{3u(x)-4u(x-h)+u(x-2h)}{2h}.
\end{align*}
Thus solving $H^h_{U,2}[u]=f$ for the reference variable, $u(x)$, leads to
\begin{align*}
u(x) = \frac{2}{3\sigma_x}	& \left[f(x)-H\left(x,\frac{-u(x+2h)+4u(x+h)-4u(x-h)+u(x-2h)}{4h}\right)\right.\\
							& \left.+\sigma_x\frac{-u(x+2h)+4u(x+h)+4u(x-h)-u(x-2h)}{4h}\right].
\end{align*}

The explicit formula of the filtered scheme can then be obtained as in the eikonal equation case using the definition of filtered scheme \eqref{defnFiltered}.
\end{example}

For the ENO schemes, we can't get an explicit formula. However, it's possible to get a fixed point iteration which has been used successfully with a fast sweeping solver in \cite{ZhangHighOrderFastSweeping}.

\subsection{Solution methods}\label{sec:solver}

The simplest solver is to use the fixed point iteration
\bq\label{FPsolver}
u^{n+1} = u^n - \alpha  (F^h[u]-f)
\eq
which corresponds to the discrete version of the parabolic equation
$u_t = - F[u] + f$
using a forward Euler step, where $F[u] = |\nabla u|$ or $F[u](x) = H(x,\nabla u)$. The fixed point iteration will be a contraction in the $l^\infty$ norm provided that we choose $\alpha$ small enough as dictated by the nonlinear CFL condition \cite{ObermanDiffSchemes}, which in the eikonal equation case means $\alpha = \bO(h)$. This will however make the solver relatively slow.

As seen in the previous section, we have explicit formulas for the upwind filtered schemes.  This allows us to use the fast sweeping method \cites{FastSweeping,FastSweepingZhao}, which is a fast iterative solution method.
Each node is updated using Gauss-Seidel iterations with alternating sweeping ordering of the domain. 
This allows information to propagate from $\Gamma$ along characteristics  to the  rest of the computational domain.
In the case of the eikonal equations, an alternative would be the Fast Marching Method \cites{FastMarching,TsitsiklisFastMarching}: the solution is constructed by using characteristic information to select the next node where the solution can be obtained.  However this requires a complicated data structure  which makes it more difficult to implement. In one dimension, the whole domain is swept with two alternating ordering of the nodes
\begin{itemize}
	\item $(i=1, \dots, N)$ and $(i=N, \dots, 1)$
\end{itemize}
which correspond to the two possible directions for the propagation of the characteristics. In two dimensions we sweep the whole domain with eight alternating ordering of the nodes
\begin{itemize}
	\item $(i=1,\dots,N,~ j=1,\dots,N)$,
	\item $(i=1,\dots,N, j=N,\dots,1)$,
	\item  \dots
	\item $(j=N,\dots,1, i=N,\dots,1)$.
\end{itemize}
corresponding respectively to up-right, up-left, down-left, down-right, right-up, left-up, left-down and right-down. Here, the first (last) four orderings help the convergence when the characteristics are aligned with the $x$-axis ($y$-axis). 

For the filtered centered and ENO schemes, we implemented the fixed point solver \eqref{FPsolver}. For the upwind filtered schemes we implemented the fast sweeping solver described above.

\subsection{Error estimates in one dimension}\label{sec:error}
In this section, we focus on the eikonal equation in one dimension, with Dirichlet boundary conditions on the endpoints of an interval. Despite the fact that the solution is Lispchitz continuous, we are able to prove, when the data $f$ is smooth enough, that the upwind schemes converge to higher order.
 This is a consequence of the fact that (i) the solution is piecewise smooth, and we can express it as a minimum  of the two ODE solutions (ii) the numerical solution is also expressed as the minimum of the left and right branches. A similar idea was used to obtain higher accuracy for conservation laws in \cite{1306.0532}.

Here we prove the higher order convergence of a particular scheme: the (unfiltered) high order upwind schemes. In this case we do not prove convergence of the filtered scheme which combined the high order upwind scheme with the monotone upwind scheme. However, we implement the filtered scheme, and we found, in practice, for the computed solution, the higher order scheme is always active.   

\begin{remark} The reason for using the filtered scheme is that it provides global stability: intermediate numerical are stable, even though in the final computed solution the accurate scheme is always active.  To use a simile, the filtered scheme acts like training wheels on a bicycle, maintaining stability even though, ultimately the training wheels do not touch the ground. \end{remark}

We consider $u$ to be the viscosity solution of the one-dimensional eikonal equation
\bq\label{u1d}
\begin{cases}
|u^\prime| = f(x), & x \in (a,b),\\
u(x) = g(x), & x \in \Gamma = \{a,b\}.
\end{cases}
\eq

To start we need first to recall the known Dynamic Programming Principle (DPP).

\begin{proposition}
Consider the dynamics
\[
\begin{cases}
\dot{y}(t) = \alpha(t) & t \in (0,+\infty),\\
y(0) = x,
\end{cases}
\]
and cost functional
\[J_x(\alpha(\cdot)) = \int_0^{t_x(\alpha)}f(y_x(s;\alpha)ds + g(y_x(t_x(\alpha),\alpha)),
\]
where $\mathcal{A} = \left\{\alpha(\cdot):[0,+\infty)\to\{-1,1\}\subset \R,  \text{ measurable}\right\}$ and $t_x$ denotes the entry time in $\Gamma$. Hence $u$ is the value function of a minimum cost problem, being given by
\bq\label{DPP}
u(x) = \inf_{\alpha \in \mathcal{A}} J_x(\alpha(\cdot)).
\eq
\end{proposition}

\begin{proof}
See \cite[Chapter IV]{BardiDolcettaOCHJBE}.
\end{proof}

We are now able to express $u$ as the minimum of two ODE solutions.

\begin{proposition}
The viscosity solution $u$ of \eqref{u1d} is given by
\bq\label{uab}
u(x):=\min\{u_a(x),u_b(x)\},
\eq
where $u_a$ and $u_b$ are respectively the solution of
\bq\label{ODEs}
\begin{cases}
u^\prime = f(x),\\
u(a) = g(a),
\end{cases}
\text{ and }
\begin{cases}
-u^\prime = f(x),\\
u(b) = g(b).
\end{cases}
\eq
\end{proposition}

\begin{proof}
Since $f>0$, the only trajectories to be consider in the minimum of \eqref{DPP} are the ones that travel straight to the endpoints $a$ and $b$. These trajectories are given by the controls $\alpha_1 \equiv -1$ and $\alpha_2 \equiv 1$, respectively. Hence
\[u(x) = \min \left\{J_x(\alpha_1(\cdot)),J_x(\alpha_2(\cdot))\right\}.\]
It's now easy to see that $u_a(x) =J_x(\alpha_1(\cdot))$ and $u_b(x) =J_x(\alpha_2(\cdot))$ and so we are done.
\end{proof}

We can now prove our result.
\begin{theorem}\label{convergence1D}
For $n \leq 6$ and if $f \in C^{(n+1)}[a,b]$ the upwind schemes are convergent. Moreover, if the solution is denoted by $u^{h,n}$, we have the following error estimate
\bq
|u^{h,n}(a+jh)-u(a+jh)| \leq Ch^nM_{n+1}
\eq
for $j = 0,\ldots,\frac{b-a}{h}$, where $C$ is a constant depending on $n$, the Lipschitz constant of $f$, $a$ and $b$ and $M_n = \max_{x\in[a,b]}|f(x)|$.
\end{theorem}

\begin{proof}
The idea of the proof consists in solving \eqref{ODEs} with backward difference schemes and realize using \eqref{uab} that we recover $u^{h,n}$, more precisely, the explicit formulas for upwind schemes discussed in subsection \ref{sec:explicit}.

Let $u^{h,n}_a$ and $u^{h,n}_b$ denote respectively the solutions obtained using backward schemes to solve \eqref{ODEs}. Hence they are the solution of
\[
\begin{cases}
U^{-,n}[u](x) = f(x)\\
u(a+jh) \text{ given for } j = 0,\ldots,n-1,
\end{cases}
\quad
\begin{cases}
-U^{+,n}[u](x) = f(x)\\
u(b-jh) \text{ given for } j = 0,\ldots,n-1.
\end{cases}
\]
Set $\tilde{u}^{h,n}(x) := \min\{u^{h,n}_a,u^{h,n}_b\}$. Under our assumptions we know that $u^{h,n}_a$ and $u^{h,n}_b$ converge respectively to $u_a$ and $u_b$ (see \cite{Quarteronietal} on multistep methods). Therefore the proof is done if we show that $\tilde{u}^{h,n}(x) = u^{h,n}(x)$.

Rather than prove this for all $n$, we give a particular example ($n=2$) and the general case should then follow easily. We will use the second order backward differentiation schemes and will therefore recover the second order upwind schemes. We have that $u_a^{h,2}$ is the solution of
\[
\frac{3u(x)-4u(x-h)+u(x-2h)}{2h} = f(x)
\]
and can therefore be written as
\[u_a^{h,2}(x) = \frac{1}{3}(4u(x-h)-u(x-2h)) + \frac{2h}{3}f(x)\]
Likewise, $u_b^{h,2}$ is the solution of
\[
-\frac{-3u(x)+4u(x+h)-u(x+2h)}{2h} = f(x)
\]
and so
\[
u_b^{h,2}(x) = \frac{1}{3}(4u(x+h)-u(x+2h)) + \frac{2h}{3}f(x)
\]
Using now \eqref{uab}, we recover \eqref{Eup} as desired.

Thus the accuracy of the numerical solution of~\eqref{u1d} is determined by the accuracy of the numerical solution of each of the two linear odes~\eqref{ODEs}.

The error estimates result naturally from the error estimates for backward difference schemes for ODEs which can be found in \cite{Quarteronietal}.
\end{proof}


\begin{remark}
The requirement $f \in C^{(n+1)}[a,b]$ is needed to obtain the order of convergence. This requirement can be relaxed to $f$ being piecewise $C^{(n+1)}$ in the same regions as the solution $u$. The idea is that we only need $u^{h,n}_a$ and $u^{h,n}_b$ to be high order convergent when they are active in the minimum of \eqref{uab}.
\end{remark}

\begin{remark}
Here we assume the exact solution is known near the boundary, but this assumption 	can be relaxed. The same order of accuracy can be obtained provided the boundary conditions are known to sufficient precision near the boundary, i.e., with the same of order of accuracy. Furthermore, these can be computed from the boundary data using standard methods.
\end{remark}

\subsection{Boundary conditions}\label{sec:boundary}
In this section we discuss the treatment of boundary conditions for the filtered scheme.

First we discuss the one dimensional case. Note that we solved the \emph{internal} problem and so the Dirichlet data is prescribed on the boundary of the computational domain. For the monotone difference method this leads to a standard application of Dirichlet boundary conditions. 

For higher order accurate methods, the situation is similar to the case of multistep methods for ordinary differential equations: more information is needed to achieve the higher accuracy.  This information can take the form of additional function values at adjacent grid points, or higher derivative information.  For practical considerations, in order to test the accuracy of the solution without introducing errors from the boundary, we extend the Dirichlet data to more grid points. More precisely, we set the exact solution (in fact, an $n$th order approximation of the exact solution is enough) at the $n$ grid points adjacent to the boundary when using the $n^{th}$ order upwind and ENO filtered schemes. Alternately, we could have used derivative information at the boundary. 

If the additional information is not available we may lose the higher accuracy. Using just the first order accurate monotone scheme reduces the order of the global accuracy. Similarly, using only the available one sided higher order approximations may decrease the accuracy since the available direction is not the one we are interested in: as we saw in the proof of Theorem \ref{convergence1D} in subsection \ref{sec:error} for the eikonal equation case, we want to interpolate towards the boundary and not away from it.

In the two-dimensional case, for the eikonal equation, we solved the \emph{external} problem and so the boundary of the computational domain did not include the Dirichlet boundary.   This poses an additional difficulty since the schemes need to be carefully defined near the boundary of the computational domain to prevent computational errors that propagate into the computational domain.  For the boundary of the computational domain we dealt as is usually done for monotone schemes: we consider only the one sided differences available.  Since the characteristics go inward, the lack of external information is not a problem here.
For the (internal) Dirichlet boundary, we proceed in the same way we did in the one-dimensional case: we set the exact solution at as many adjacent grid points of the boundary as needed depending on the order of accuracy of the scheme used. 

For general Hamilton-Jacobi equations, the computational boundary can cause problems, depending on the discretization used. For the Godunov scheme, which reduces to~\eqref{monotone2deikonal} in the case of the eikonal equation, there are no problems, so this is what we used for the eikonal equation.  However, for general Hamilton-Jacobi equations in two dimensions using the Lax-Friedrichs schemes~\eqref{monotone2dHJ} with high order interpolation is more complicated~\cite{ZhangHighOrderFastSweeping}, and can lead to errors at the computational boundary.

\section{Computational Results}
\subsection{Example solutions in one dimension}\label{sec:examples1D}
In this subsection we discuss the examples considered in one dimension. In all of them the solution is piecewise smooth with a single singularity. Their purpose is confirm the improved accuracy of the filtered schemes, as well as the high order convergence of the upwind schemes for the eikonal equation. All examples are displayed in Figure \ref{fig:ExProfiles}.

\begin{figure}[htdp]
\centering
\begin{tabular}{ccc}
\includegraphics[width=0.3\textwidth]{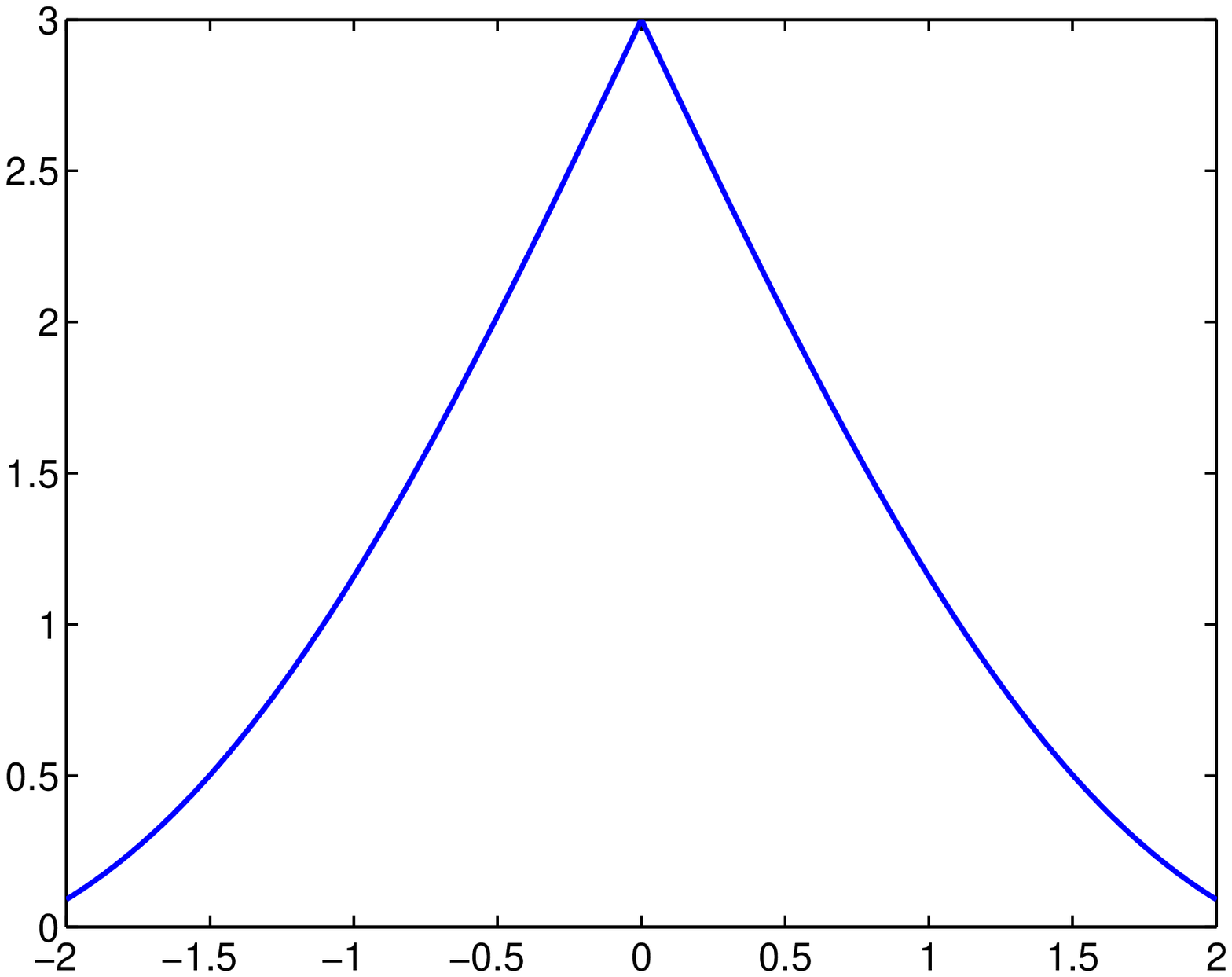} & \includegraphics[width=0.3\textwidth]{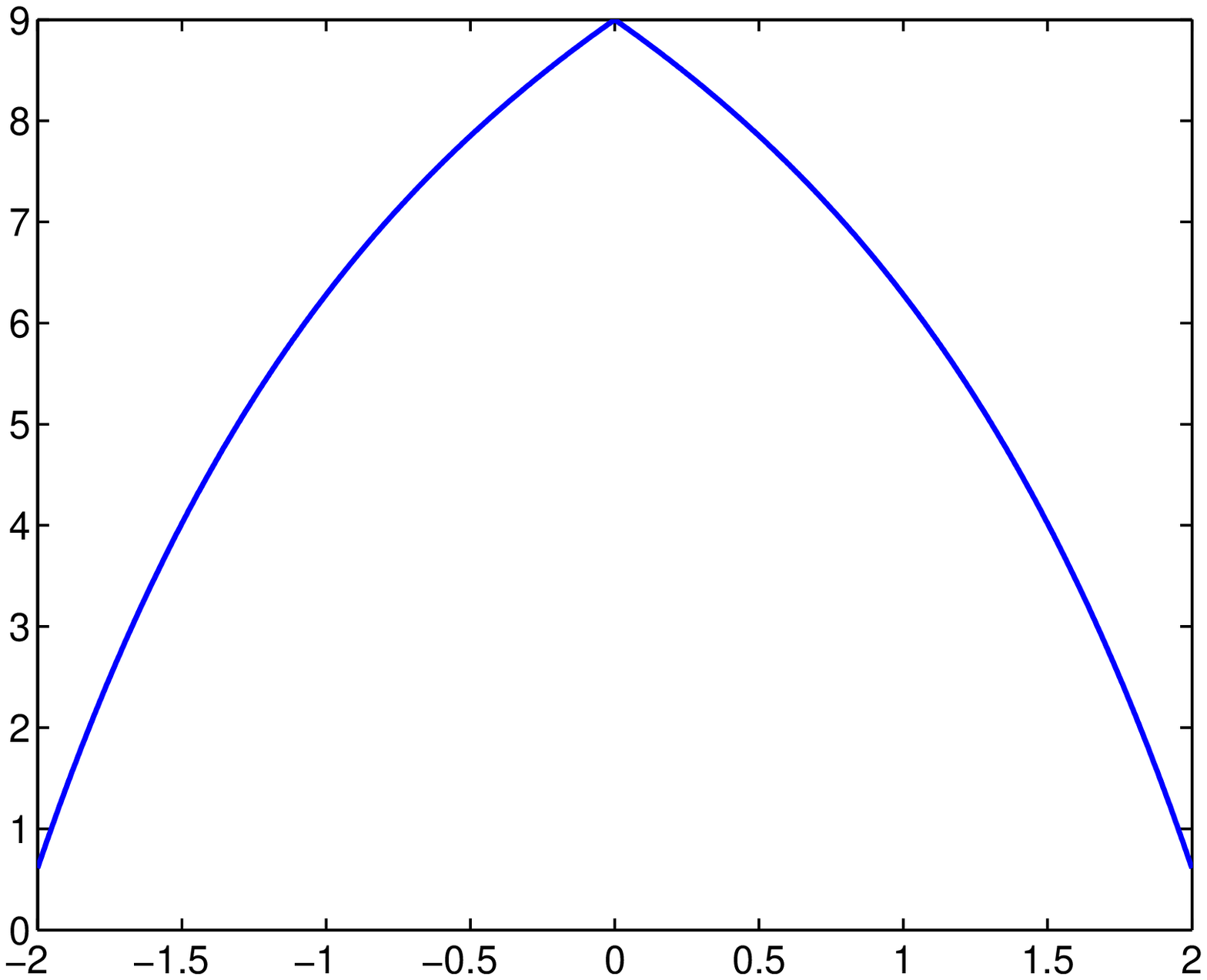} & \includegraphics[width=0.3\textwidth]{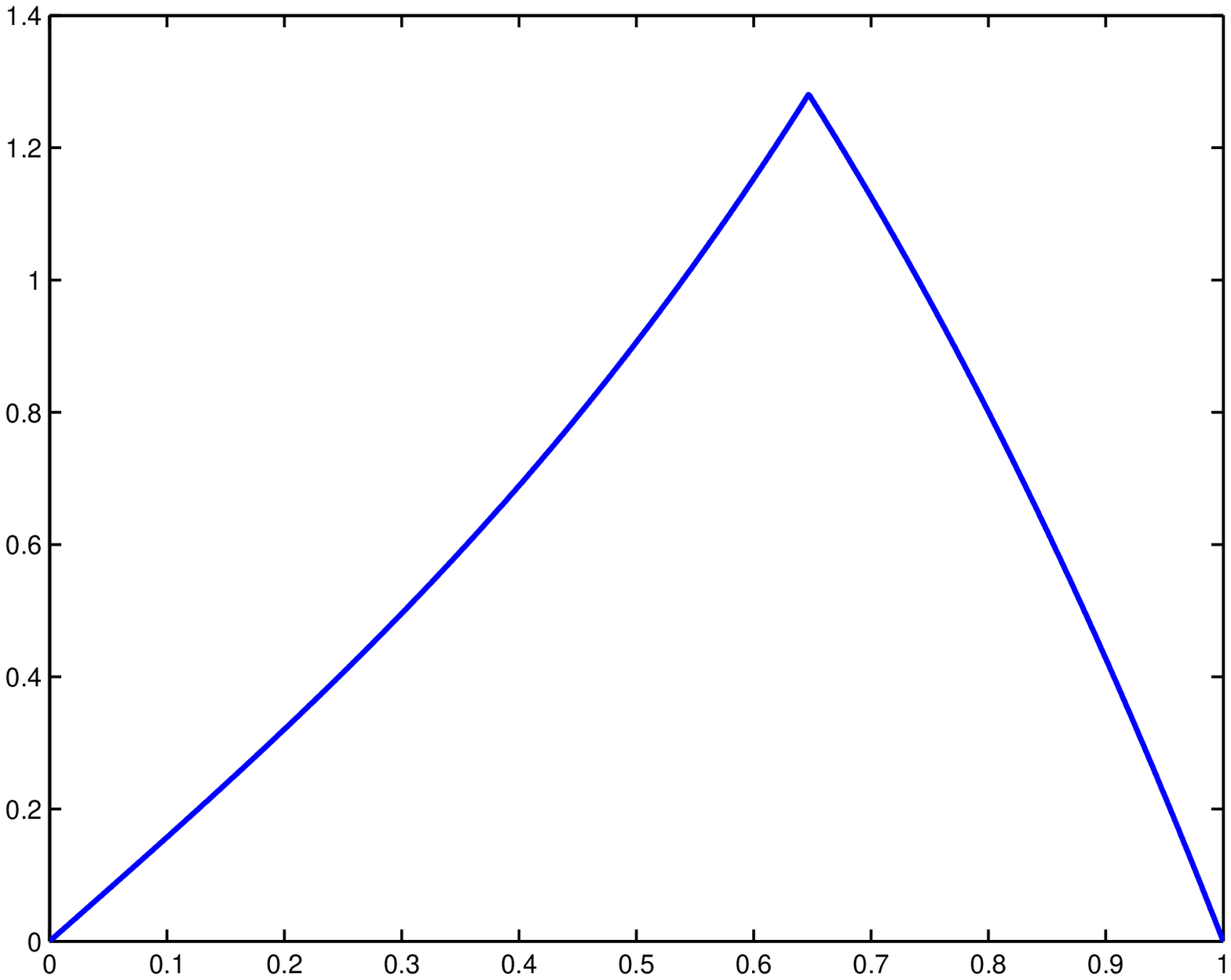}\\
\end{tabular}
\begin{tabular}{cc}
\includegraphics[width=0.3\textwidth]{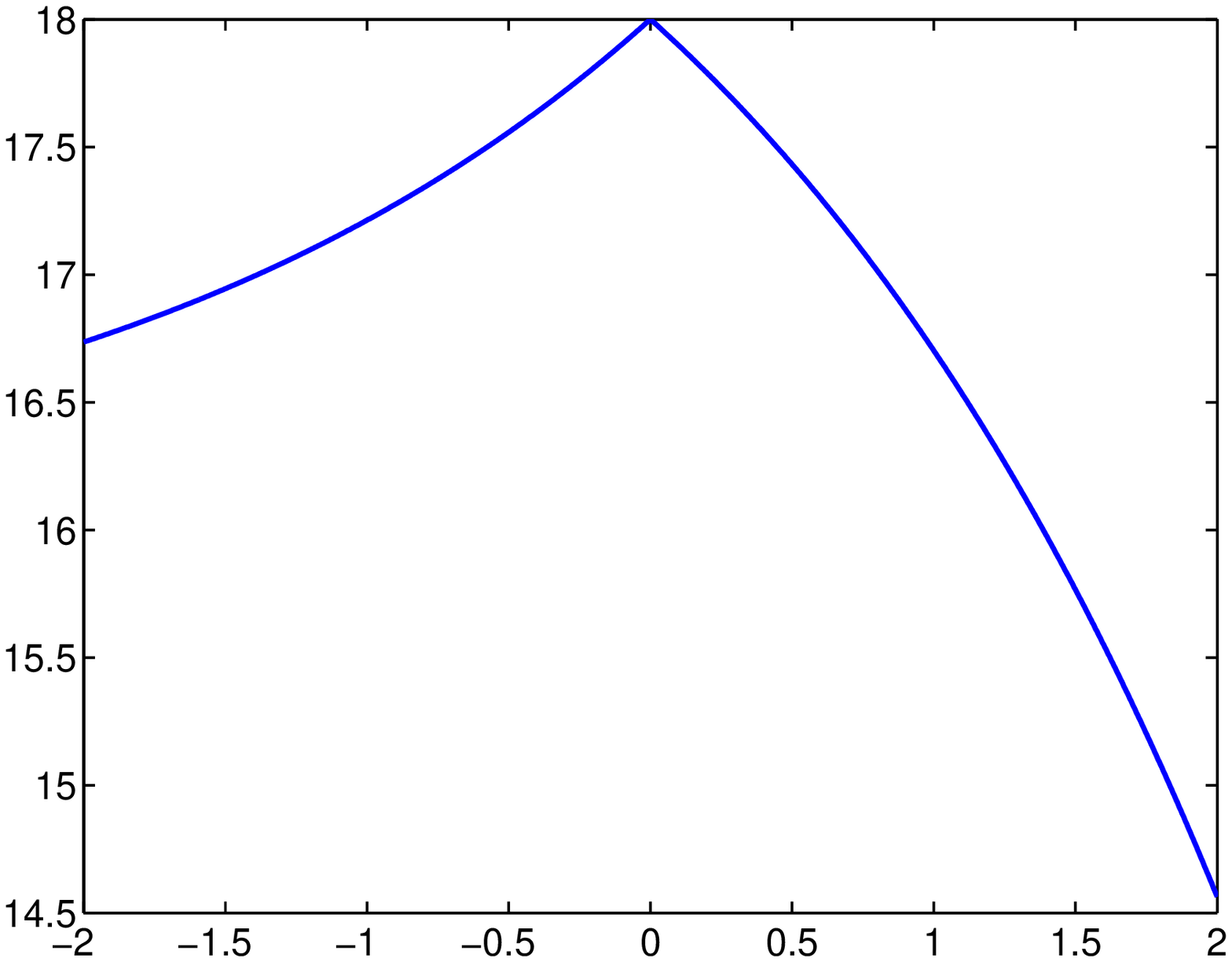} & \includegraphics[width=0.3\textwidth]{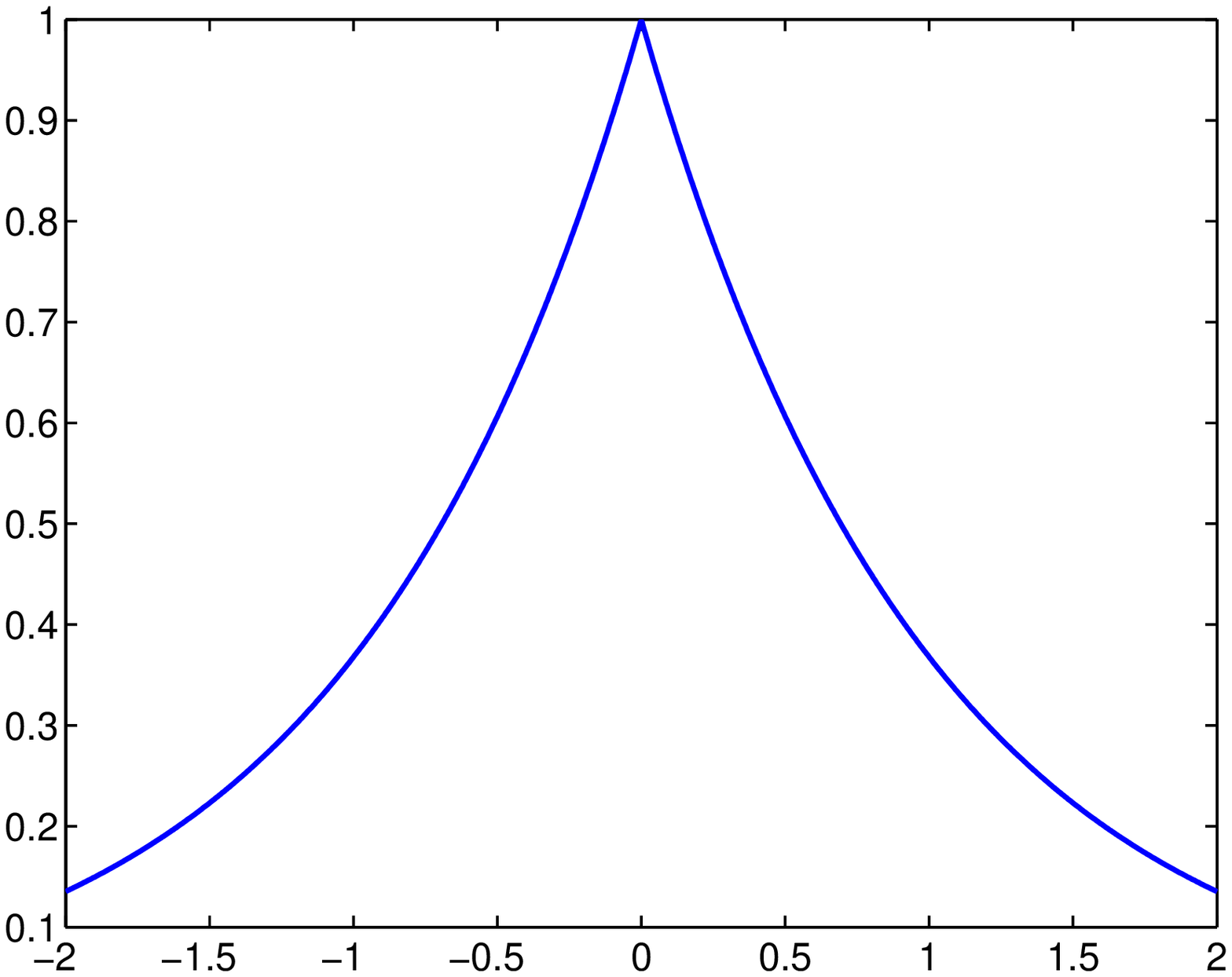}\\
\end{tabular}
\caption{Profile of the solutions of the five examples considered in one dimension (at the top, eikonal equation examples, at the bottom, HJ equations examples).}
\label{fig:ExProfiles}
\end{figure}

The first example is the eikonal equation with $f(x) = 1+\cos(x)$ with the Dirichlet boundary conditions being prescribed at $x = \pm 2$. The computational domain is $[-2,2]$. The exact solution is given by $u(x) = 3-|x+\sin(x)|$ and it's therefore piecewise smooth with a singularity at $x=0$ (see Figure \ref{fig:ExProfiles}). We represent the solution obtained with the monotone scheme and the $2^{nd}$ upwind filtered scheme for $50$ mesh points near the singularity in $[-0.4,0.4]$ on Figure \ref{fig:Ex1MonotoneVs2ndupwind}.

\newcommand{\ww}{0.65}
\begin{figure}[htdp]
\centering
\includegraphics[width=\ww\textwidth]{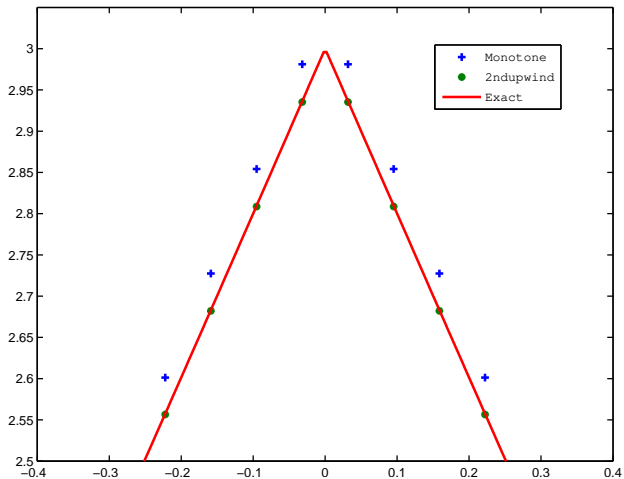}
\caption{Exact solution and solutions obtained with the monotone scheme and the $2^{nd}$ order upwind filtered scheme with $50$ grid points for the first example of the eikonal equation.}
\label{fig:Ex1MonotoneVs2ndupwind}
\end{figure}

The second example is again the eikonal equation with $f(x) = 1+e^{|x|}$, where the Dirichlet boundary conditions are once again prescribed at $x = \pm 2$ and the computational domain is $[-2,2]$. The exact solution is given by $u(x) = 10- |x|-e^{|x|}$ and as in the previous example, it's piecewise smooth with a singularity at $x=0$ (see Figure \ref{fig:ExProfiles}).

The third example, also a solution of the eikonal equation, is given by
\[u(x) = \begin{cases}
x^3+ax		& x \in [0,x_0],\\
1+a-ax-x^3	& x \in [x_0,1],
\end{cases}\]
with $a = \frac{1-2x_0^3}{2x_0 -1}$, $x_0 = \frac{\sqrt[3]{2}+2}{4\sqrt[3]{2}}$ and therefore $f(x) = 3x^2+a$. This example was chosen for two main reasons: there is no symmetry in the relationship between the singularity and the grid points, as opposed to the two previous examples where the singularity was always a midpoint of two consecutive grid points; this is one the examples in \cite{Abgrall} that the author uses to check the rate of convergence of the proposed method. The difference is that in \cite{Abgrall} the error in the $l^\infty$ norm is computed at the grid points in the interval $\left[\frac{1}{\sqrt[3]{2}},\frac{1}{2}\right]$, instead of all the grid points as we do here. The author chooses that interval since it's an interval where the solution is smooth but as we explained above we can look at the error on all grid points and still obtain the high order convergence.

We consider as well two HJ equations. The first one given by $H(p) = p^2$, a convex Hamiltonian, with $f(x) = e^x$ and
\[u(x) =\begin{cases}
-2e^{\frac{x}{2}}+20	& x \in [-2,0],\\
2e^{\frac{x}{2}}+16		& x \in [0,2].
\end{cases}\]
The second one given by $H(p) = \cos(p)^2+|p|$, a nonconvex Hamiltonian considered in \cite{Abgrall}, with $u(x) = e^{-|x|}$ and $f(x) = \cos(e^{-|x|})^2 + e^{-|x|}$. The profile of both solutions is depicted in Figure \ref{fig:ExProfiles} and the Dirichlet boundary conditions are prescribed at $x = \pm 2$, with the computational domain being $[-2,2]$. For the nonconvex example, the factor $\sqrt{h}$ in the filtered scheme \eqref{defnFiltered} was replaced by $h^{\frac{1}{10}}$ see Remark~\ref{remark:rates}.

The computational domain is discretized on a grid with $N$ points and the singularity is never a grid point.

\subsection{Computational results in one dimension}\label{sec:discussion1D}
In this subsection we discuss the computational results obtained in one dimension. The main purpose is to demonstrate that the filtered scheme achieves the higher order accuracy and that, in particular for the eikonal equation, the upwind filtered schemes achieve higher order convergence rate as proved in subsection \ref{sec:error} for the (unfiltered) upwind schemes. We organize the discussion in three parts: accuracy and behavior, order of convergence and upwind vs ENO. For the eikonal equation, we obtained results with the monotone scheme \eqref{monotone1deikonal} and the respective filtered schemes using as the accurate scheme the second centered scheme and the second, third and forth order upwind and ENO schemes. For HJ equations, we obtain results using the monotone scheme \eqref{monotone1dHJ} and the respective filtered schemes using as the accurate scheme the second order centered, upwind and ENO schemes. Third order upwind and ENO filtered schemes were also used, but they didn't show any advantage over the second order schemes.

\begin{table}[htdp]
\centering\footnotesize
\begin{tabular}{ccccccccc}
\multicolumn{9}{c}{Errors and order, $1^{st}$ Example}\\
N & \multicolumn{2}{c}{Monotone} & \multicolumn{2}{c}{$2^{nd}$ Upwind} & \multicolumn{2}{c}{$3^{rd}$ Upwind} & \multicolumn{2}{c}{$4^{th}$ Upwind}\\
\hline
64 & \num{4.465e-02} & - & \num{1.141e-03} & - & \num{8.532e-05} & - & \num{2.646e-06} & - \\
128 & \num{2.223e-02} & 0.99 & \num{2.908e-04} & 1.95 & \num{1.076e-05} & 2.95 & \num{1.700e-07} & 3.92 \\
256 & \num{1.109e-02} & 1.00 & \num{7.337e-05} & 1.98 & \num{1.348e-06} & 2.98 & \num{1.074e-08} & 3.96 \\
512 & \num{5.538e-03} & 1.00 & \num{1.842e-05} & 1.99 & \num{1.687e-07} & 2.99 & \num{6.745e-10} & 3.98 \\
1024 & \num{2.767e-03} & 1.00 & \num{4.615e-06} & 1.99 & \num{2.109e-08} & 3.00 & \num{4.224e-11} & 3.99 \\ \\
N & \multicolumn{2}{c}{$2^{nd}$ centered} & \multicolumn{2}{c}{$2^{nd}$ ENO} & \multicolumn{2}{c}{$3^{rd}$ ENO} & \multicolumn{2}{c}{$4^{th}$ ENO}\\
\hline
64 & \num{6.553e-04} & - & \num{7.660e-04} & - & \num{2.780e-05} & - & \num{5.561e-07} & - \\
128 & \num{1.559e-04} & 2.05 & \num{1.918e-04} & 1.97 & \num{3.546e-06} & 2.94 & \num{3.544e-08} & 3.93 \\
256 & \num{3.789e-05} & 2.03 & \num{4.803e-05} & 1.99 & \num{4.470e-07} & 2.97 & \num{2.236e-09} & 3.96 \\
512 & \num{9.451e-06} & 2.00 & \num{1.201e-05} & 1.99 & \num{5.608e-08} & 2.99 & \num{1.404e-10} & 3.98 \\
1024 & \num{2.317e-06} & 2.03 & \num{3.004e-06} & 2.00 & \num{7.022e-09} & 2.99 & \num{8.776e-12} & 3.99 \\\end{tabular}
\caption{Accuracy in the $l^\infty$ norm and order of convergence of the schemes for the first example of the eikonal equation.}
\label{table:Ex1errors}
\end{table}

\begin{table}[htdp]
\centering\footnotesize
\begin{tabular}{ccccccccc}
\multicolumn{9}{c}{Errors and order, $2^{nd}$ Example}\\
N & \multicolumn{2}{c}{Monotone} & \multicolumn{2}{c}{$2^{nd}$ Upwind} & \multicolumn{2}{c}{$3^{rd}$ Upwind} & \multicolumn{2}{c}{$4^{th}$ Upwind}\\
\hline
64 & \num{1.997e-01} & - & \num{8.011e-03} & - & \num{3.642e-04} & - & \num{1.766e-05} & - \\
128 & \num{9.984e-02} & 0.99 & \num{2.042e-03} & 1.95 & \num{4.716e-05} & 2.92 & \num{1.162e-06} & 3.88 \\
256 & \num{4.992e-02} & 0.99 & \num{5.153e-04} & 1.98 & \num{5.995e-06} & 2.96 & \num{7.441e-08} & 3.94 \\
512 & \num{2.496e-02} & 1.00 & \num{1.294e-04} & 1.99 & \num{7.555e-07} & 2.98 & \num{4.706e-09} & 3.97 \\
1024 & \num{1.248e-02} & 1.00 & \num{3.242e-05} & 1.99 & \num{9.482e-08} & 2.99 & \num{2.959e-10} & 3.99 \\ \\
N & \multicolumn{2}{c}{$2^{nd}$ centered} & \multicolumn{2}{c}{$2^{nd}$ ENO} & \multicolumn{2}{c}{$3^{rd}$ ENO} & \multicolumn{2}{c}{$4^{th}$ ENO}\\
\hline
64 & \num{6.358e-03} & - & \num{3.983e-03} & - & \num{1.705e-03} & - & \num{1.492e-03} & - \\
128 & \num{1.570e-03} & 2.00 & \num{1.018e-03} & 1.95 & \num{2.764e-04} & 2.60 & \num{1.823e-03} & -0.29 \\
256 & \num{3.859e-04} & 2.01 & \num{2.573e-04} & 1.97 & \num{4.899e-05} & 2.48 & \num{2.499e-04} & 2.85 \\
512 & \num{9.700e-05} & 1.99 & \num{6.466e-05} & 1.99 & \num{8.981e-06} & 2.44 & \num{5.037e-05} & 2.30 \\
1024 & \num{2.410e-05} & 2.01 & \num{1.621e-05} & 1.99 & \num{1.422e-06} & 2.66 & \num{4.332e-05} & 0.22 \\\end{tabular}
\caption{Accuracy in the $l^\infty$ norm and order of convergence of the schemes for the second example of the eikonal equation.}
\label{table:Ex2errors}
\end{table}

\begin{table}[htdp]
\centering\footnotesize
\begin{tabular}{ccccccccc}
\multicolumn{9}{c}{Errors and order, $3^{rd}$ Example}\\
N & \multicolumn{2}{c}{Monotone} & \multicolumn{2}{c}{$2^{nd}$ Upwind} & \multicolumn{2}{c}{$3^{rd}$ Upwind} & \multicolumn{2}{c}{$4^{th}$ Upwind}\\
\hline
64 & \num{1.368e-02} & - & \num{3.079e-04} & - & \num{1.332e-15} & - & \num{2.887e-15} & - \\
128 & \num{6.756e-03} & 1.01 & \num{7.860e-05} & 1.95 & \num{1.110e-15} & 0.26 & \num{3.109e-15} & -0.11 \\
256 & \num{3.417e-03} & 0.98 & \num{1.960e-05} & 1.99 & \num{2.220e-15} & -0.99 & \num{4.219e-15} & -0.44 \\
512 & \num{1.703e-03} & 1.00 & \num{4.924e-06} & 1.99 & \num{2.887e-15} & -0.38 & \num{2.665e-15} & 0.66 \\
1024 & \num{8.521e-04} & 1.00 & \num{1.232e-06} & 2.00 & \num{5.107e-15} & -0.82 & \num{4.663e-15} & -0.81 \\ \\
N & \multicolumn{2}{c}{$2^{nd}$ centered} & \multicolumn{2}{c}{$2^{nd}$ ENO} & \multicolumn{2}{c}{$3^{rd}$ ENO} & \multicolumn{2}{c}{$4^{th}$ ENO}\\
\hline
64 & \num{6.357e-04} & - & \num{3.079e-04} & - & \num{2.442e-15} & - & \num{1.332e-15} & - \\
128 & \num{1.596e-04} & 1.97 & \num{7.860e-05} & 1.95 & \num{7.550e-15} & -1.61 & \num{4.441e-15} & -1.72 \\
256 & \num{3.950e-05} & 2.00 & \num{1.960e-05} & 1.99 & \num{2.109e-14} & -1.47 & \num{3.819e-14} & -3.09 \\
512 & \num{9.886e-06} & 1.99 & \num{4.924e-06} & 1.99 & \num{3.220e-14} & -0.61 & \num{1.134e-13} & -1.57 \\
1024 & \num{2.192e-06} & 2.17 & \num{1.232e-06} & 2.00 & \num{5.818e-14} & -0.85 & \num{8.527e-14} & 0.41 \\\end{tabular}
\caption{Accuracy in the $l^\infty$ norm and order of convergence of the schemes for the third example of the eikonal equation.}
\label{table:Ex3errors}
\end{table}

\begin{table}[htdp]
\centering\footnotesize
\begin{tabular}{ccccccccc}
\multicolumn{9}{c}{Errors and order, $4^{th}$ Example}\\
N & \multicolumn{2}{c}{Monotone} & \multicolumn{2}{c}{$2^{nd}$ centered} & \multicolumn{2}{c}{$2^{nd}$ Upwind} & \multicolumn{2}{c}{$2^{nd}$ ENO}\\
\hline
64 & \num{1.234e-01} & - & \num{8.532e-02} & - & \num{9.307e-02} & - & \num{8.308e-02} & - \\
128 & \num{6.106e-02} & 1.00 & \num{4.226e-02} & 1.00 & \num{4.179e-02} & 1.14 & \num{4.132e-02} & 1.00 \\
256 & \num{3.037e-02} & 1.00 & \num{2.108e-02} & 1.00 & \num{2.095e-02} & 0.99 & \num{2.067e-02} & 0.99 \\
512 & \num{1.515e-02} & 1.00 & \num{1.054e-02} & 1.00 & \num{1.044e-02} & 1.00 & \num{1.057e-02} & 0.97 \\
1024 & \num{7.563e-03} & 1.00 & \num{5.310e-03} & 0.99 & \num{5.304e-03} & 0.98 & \num{5.272e-03} & 1.00 \\
\end{tabular}
\caption{Accuracy in the $l^\infty$ norm and order of convergence of the schemes for the fourth example ($H(p)=p^2$).}
\label{table:Ex4errors}
\end{table}

\begin{table}[htdp]
\centering\footnotesize
\begin{tabular}{ccccccccc}
\multicolumn{9}{c}{Errors and order, $5^{th}$ Example}\\
N & \multicolumn{2}{c}{Monotone} & \multicolumn{2}{c}{$2^{nd}$ centered} & \multicolumn{2}{c}{$2^{nd}$ Upwind} & \multicolumn{2}{c}{$2^{nd}$ ENO}\\
\hline
64 & \num{1.328e-01} & - & \num{2.105e-02} & -- & \num{1.129e-01} & - & \num{8.577e-02} & - \\
128 & \num{1.095e-01} & 0.27 & \num{1.111e-02} & 0.91 & \num{3.446e-02} & 1.69 & \num{6.995e-02} & 0.29 \\
256 & \num{8.855e-02} & 0.31 & \num{4.365e-03} & 1.34 & \num{1.379e-02} & 1.31 & \num{5.072e-02} & 0.46 \\
512 & \num{7.043e-02} & 0.33 & \num{2.360e-03} & 0.88 & \num{3.772e-03} & 1.86 & \num{1.288e-02} & 1.97 \\
1024 & \num{5.401e-02} & 0.38 & \num{2.693e-03} & -0.19 & \num{1.870e-03} & 1.01 & \num{8.170e-03} & 0.66 \\\end{tabular}
\caption{Accuracy in the $l^\infty$ norm and order of convergence of the schemes for the fifth example ($H(p)=\cos(p)^2+|p|$).}
\label{table:Ex5errors}
\end{table}

\textit{Accuracy and behavior of the filtered schemes.} We begin by comparing the accuracy of the monotone scheme with the filtered schemes by looking at the error in the $l^\infty$ norm in Figure \ref{fig:Exloglog1D} and Tables \ref{table:Ex1errors}, \ref{table:Ex2errors}, \ref{table:Ex3errors}, \ref{table:Ex4errors}, \ref{table:Ex5errors}. As expected the filtered schemes have improved accuracy.

Once close to the solution, the filtered schemes behave as designed choosing to use the accurate scheme whenever possible, i.e., whenever they interpolate the solution in a smooth region. Therefore, in the eikonal equation case, the monotone scheme ends up not being used in the upwind and ENO filtered schemes since these schemes have a choice on where to interpolate, choosing to always do so on the region where the solution is smooth. This isn't however the case when the $2^{nd}$ order centered scheme is used as the accurate scheme. In this case, the filtered scheme falls back to the monotone scheme on the two grid points adjacent to the singularity. As for the HJ equations case, the forward and backward approximation are both always used and thus near the singularity the filtered schemes fall back to the monotone scheme.

\begin{figure}[htdp]
\centering
\begin{tabular}{c}
\includegraphics[width=0.65\textwidth]{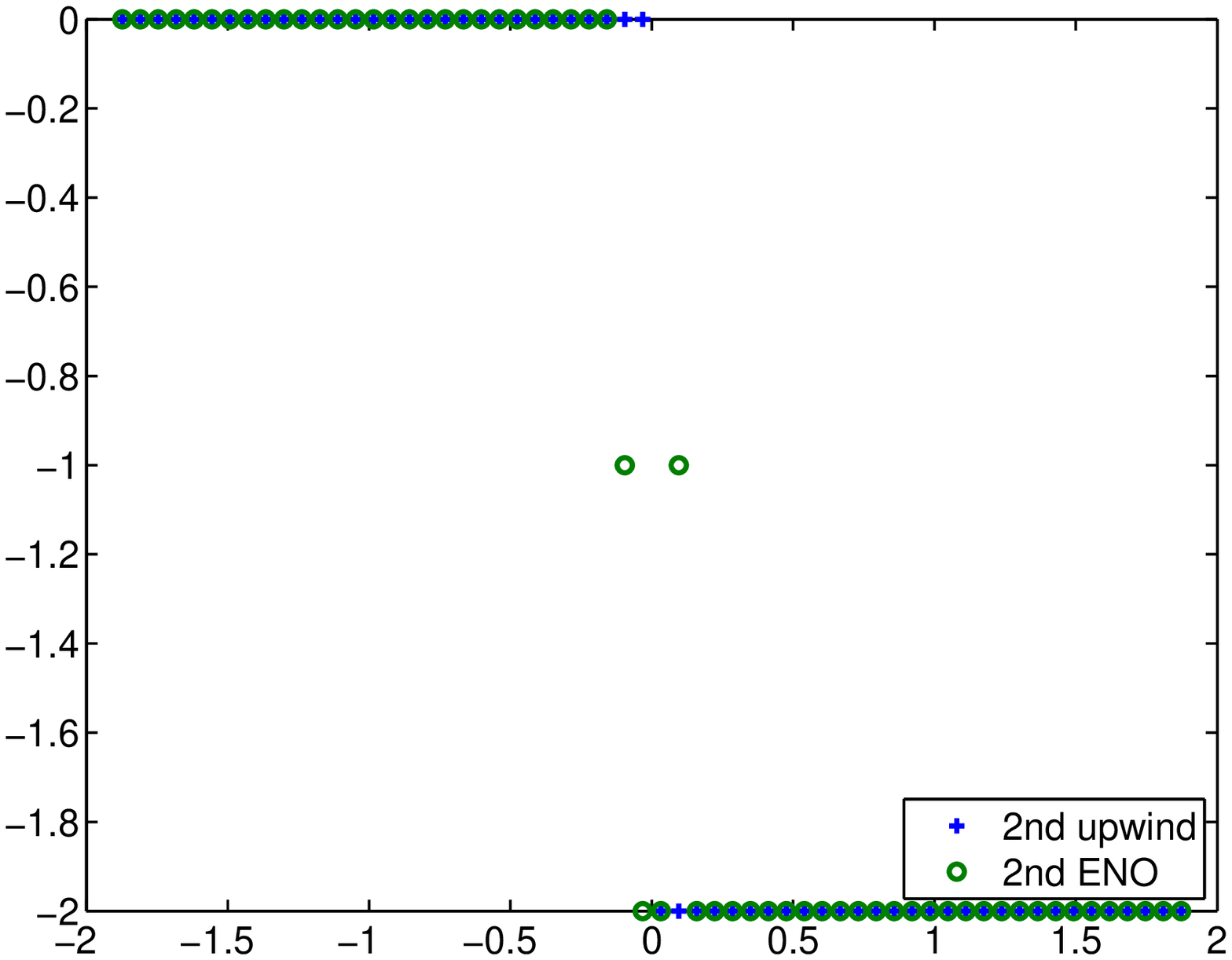}\\
\includegraphics[width=0.65\textwidth]{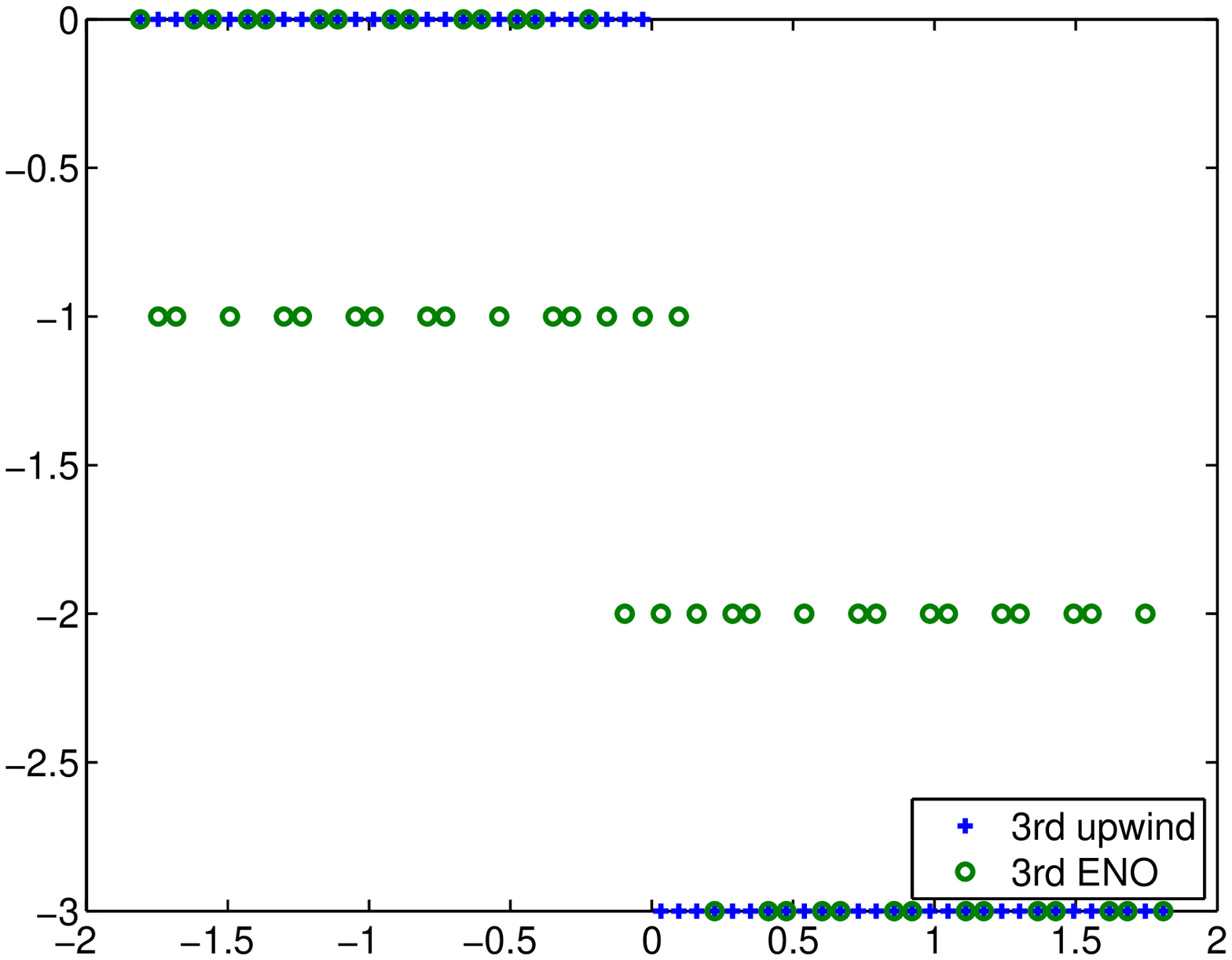}\\
\includegraphics[width=0.65\textwidth]{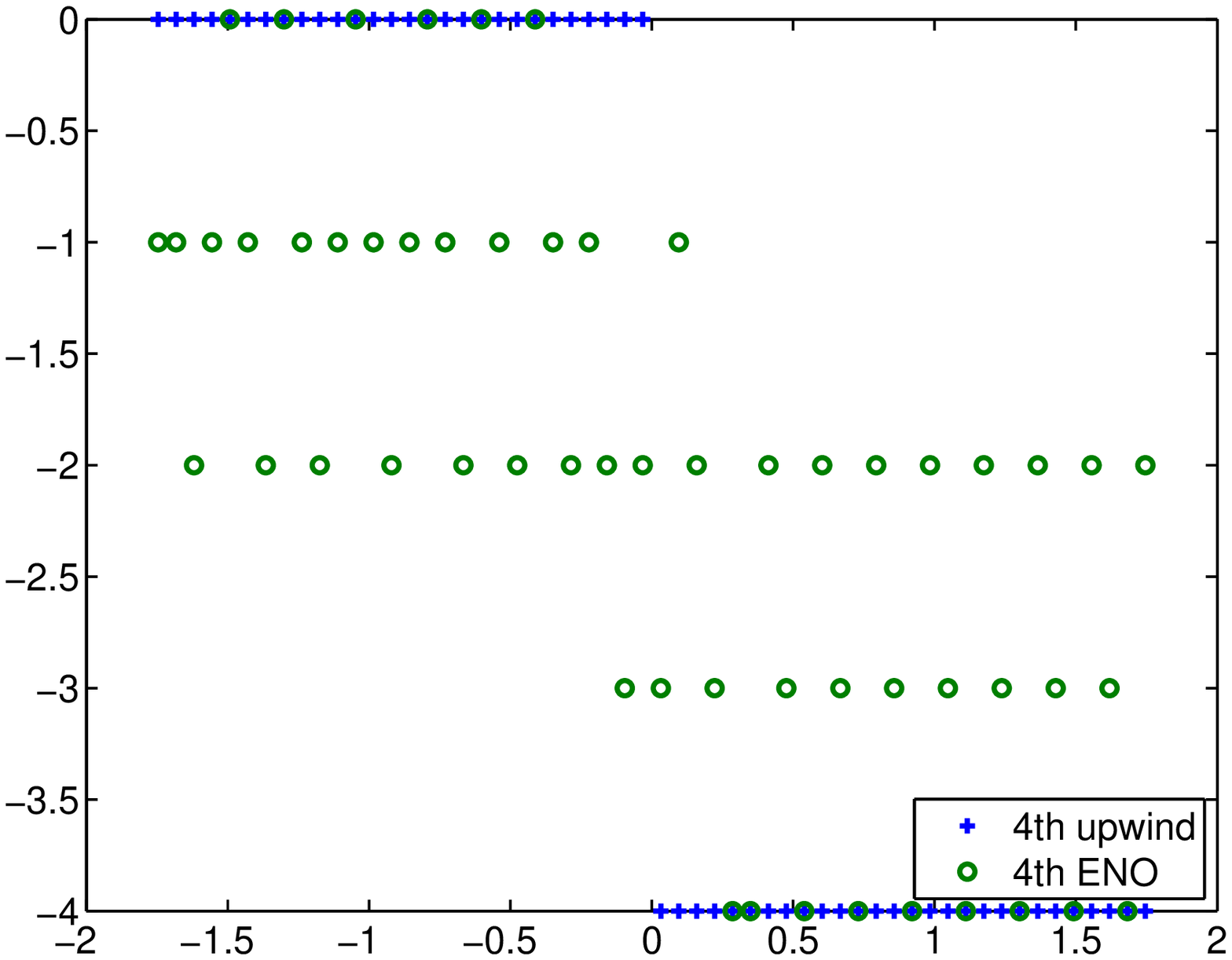}
\end{tabular}
\caption{Active stencils in the accurate scheme in the last iteration for the solutions of the second example considered: $-i$ means that $i$ points to the left were used in the interpolation.
}
\label{fig:Ex2ActiveStencils}
\end{figure}

\subsubsection*{Order of convergence.} We first discuss the eikonal equation case. Examining Figure \ref{fig:Exloglog1D} and Tables \ref{table:Ex1errors}, \ref{table:Ex2errors}, \ref{table:Ex3errors}, we conclude that all the upwind filtered schemes have convergence rate corresponding to the order of accuracy of the accurate scheme, except in the last example, where for the $3^{rd}$ and $4^{th}$ order schemes we obtain machine accuracy. This exception is explained by the fact that in this example the solution is piecewise cubic and therefore these schemes end up being exact (interpolating a cubic polynomial with 4 or more points yields the exact same cubic polynomial). Obtaining the higher order convergence rate is in accordance with Theorem \ref{convergence1D} since for the upwind filtered schemes the accurate scheme is always active as mentioned above. We should point out that this higher rate of convergence was already possible to obtain using ENO schemes as is depicted in Figure \ref{fig:Exloglog1D} (with the sole exception of the $4^{th}$ order ENO scheme in the second example, which we discuss below). Moreover, the filtered scheme using the second centered scheme also provided second order convergence even though as pointed above it falls into the monotone scheme near the singularity, more precisely on the two grid points that enclose it.

\begin{figure}[htdp]
\centering
\begin{tabular}{c}
\includegraphics[width=0.65\textwidth]{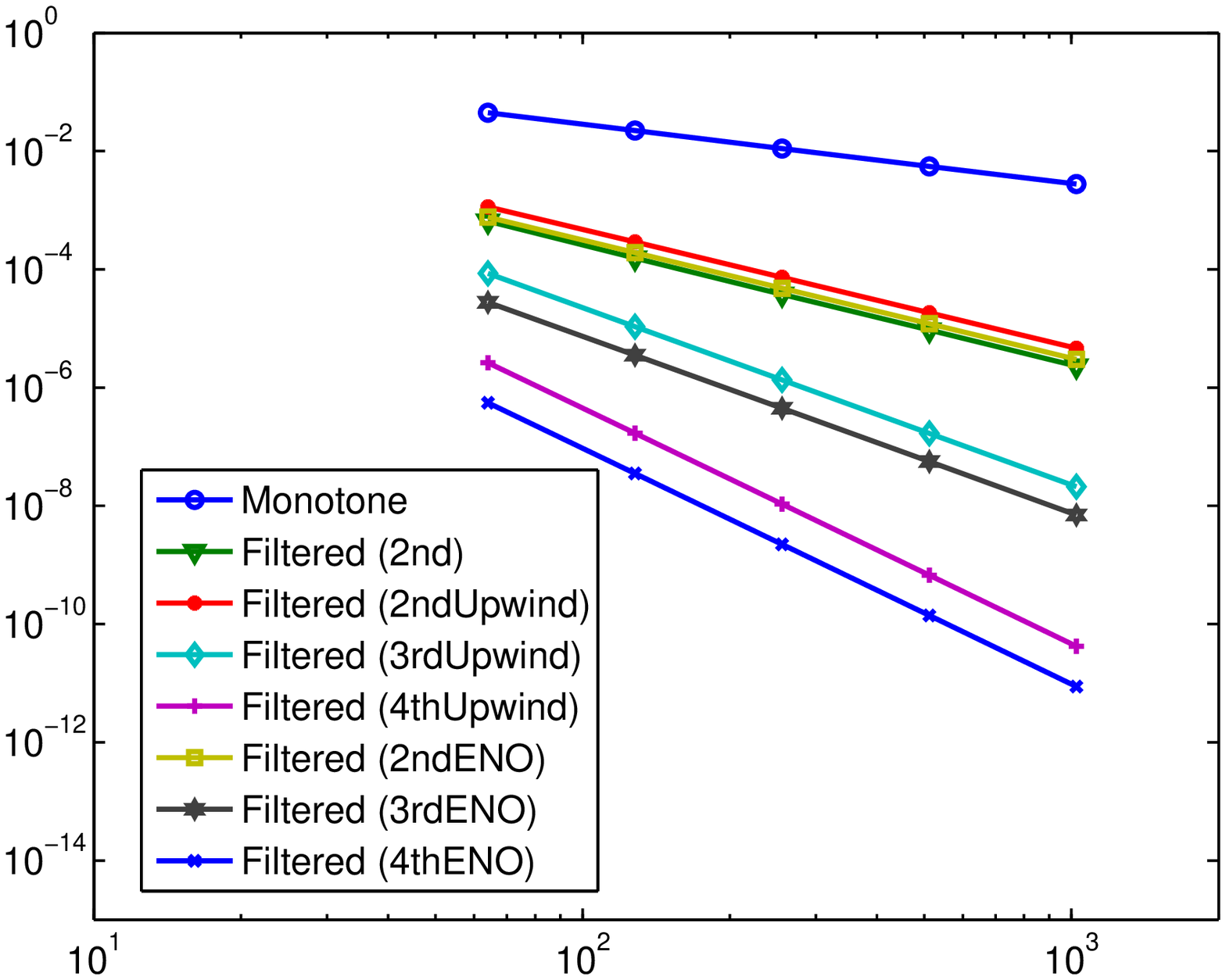}\\
\includegraphics[width=0.65\textwidth]{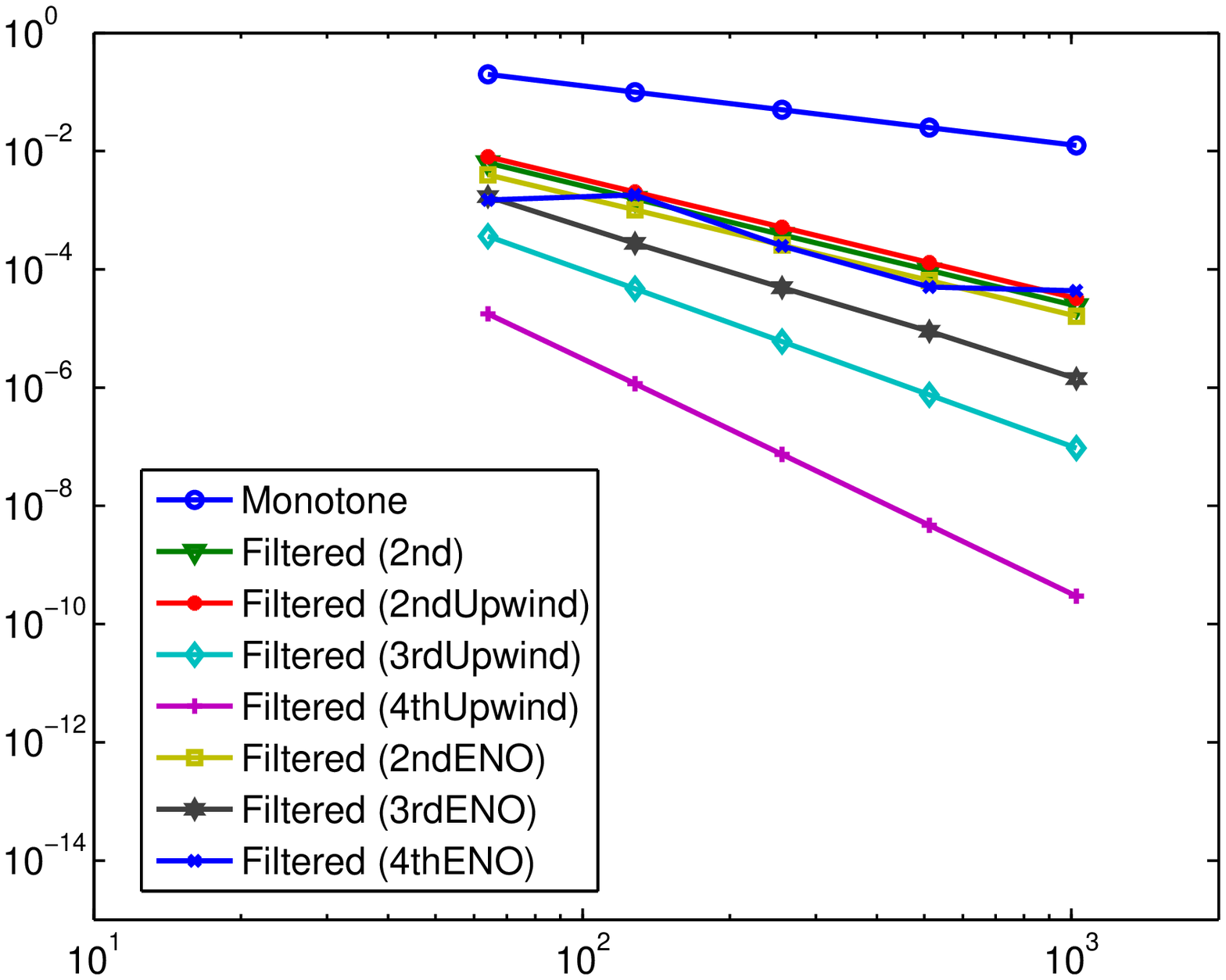}\\
\includegraphics[width=0.65\textwidth]{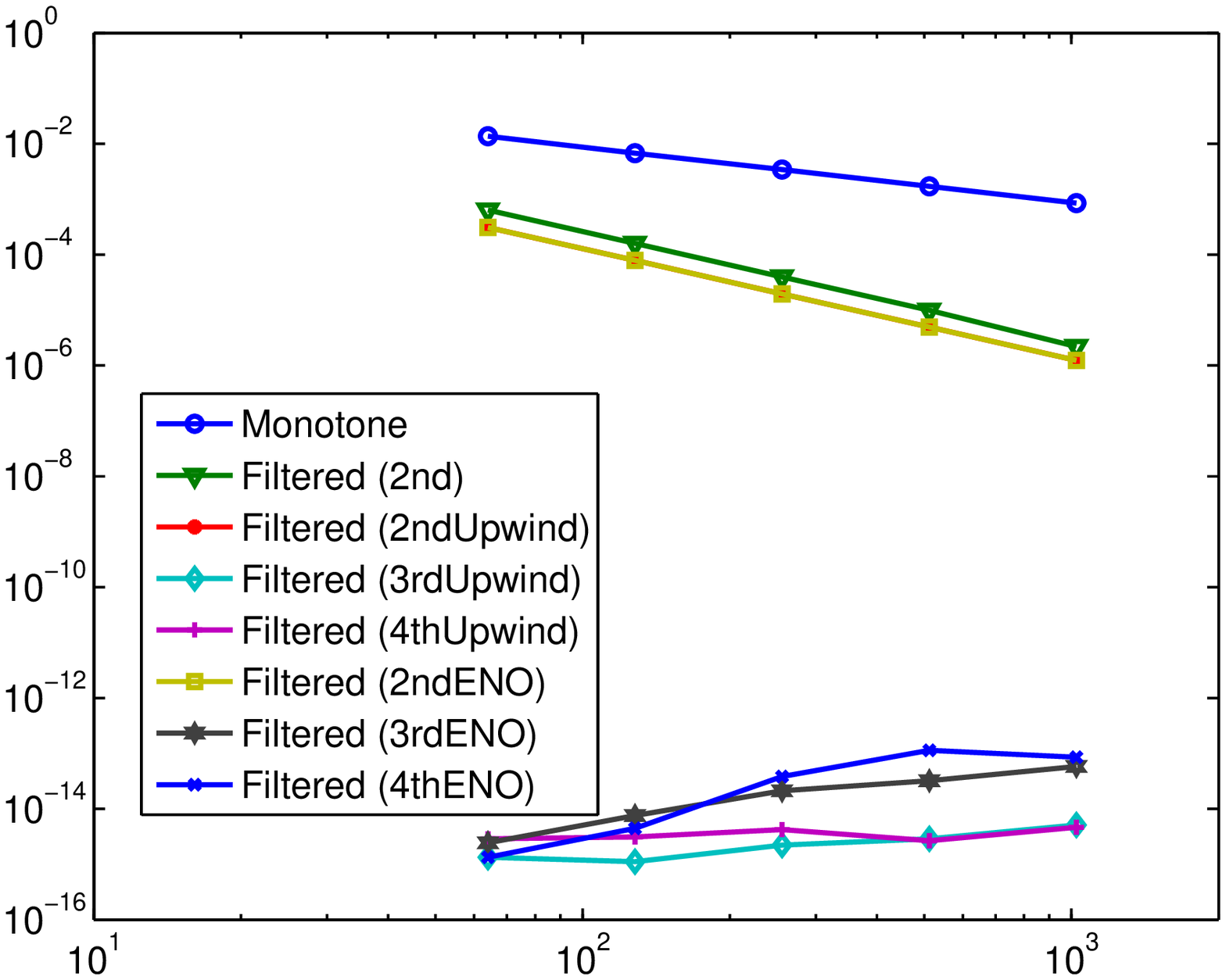}
\end{tabular}
\caption{Log-log plot of the errors for the one-dimensional examples of the eikonal equation.}
\label{fig:Exloglog1D}
\end{figure}

\begin{figure}[htdp]
\centering
\begin{tabular}{cc}
\hspace{-1in}
\includegraphics[width=0.65\textwidth]{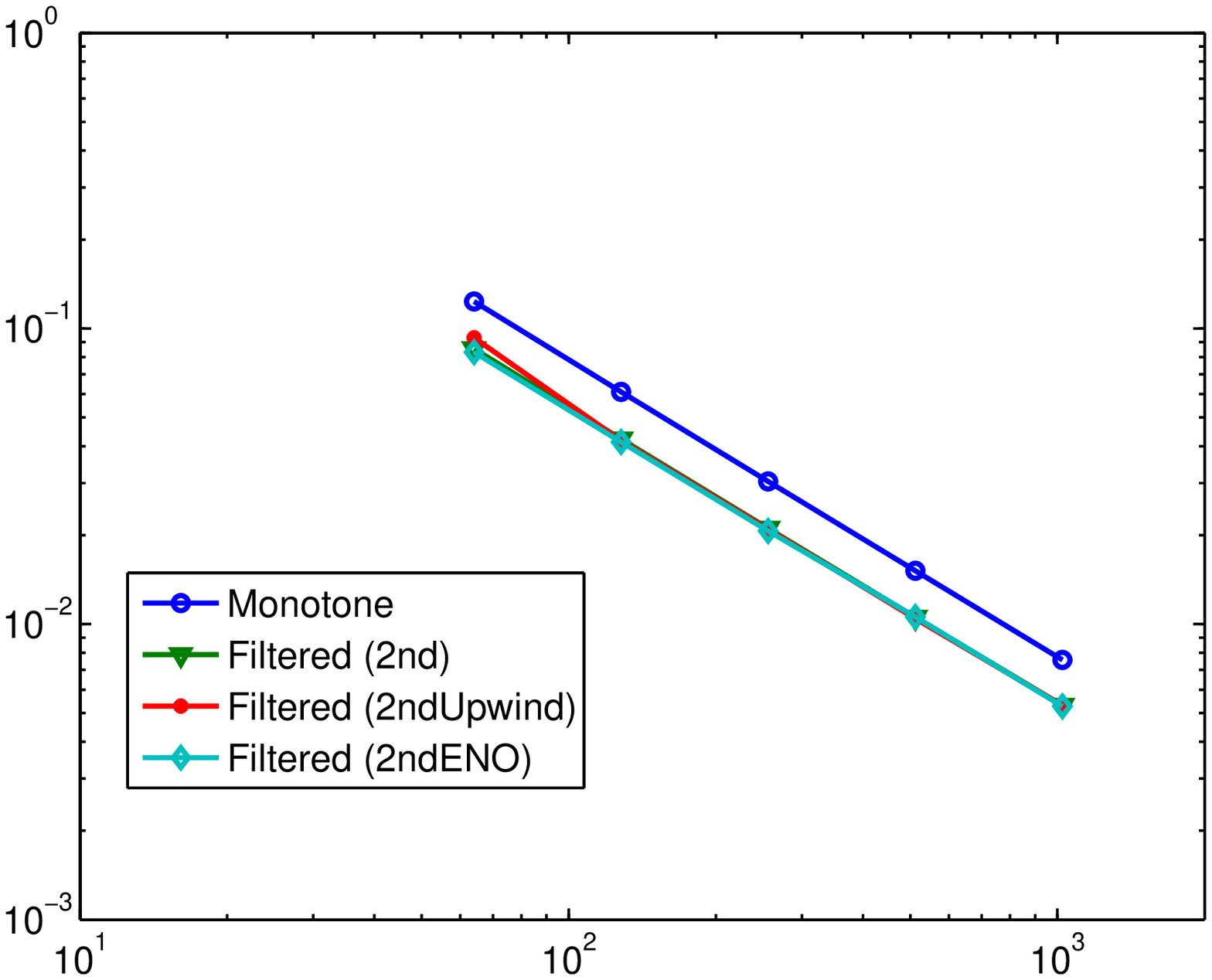} & \includegraphics[width=0.65\textwidth]{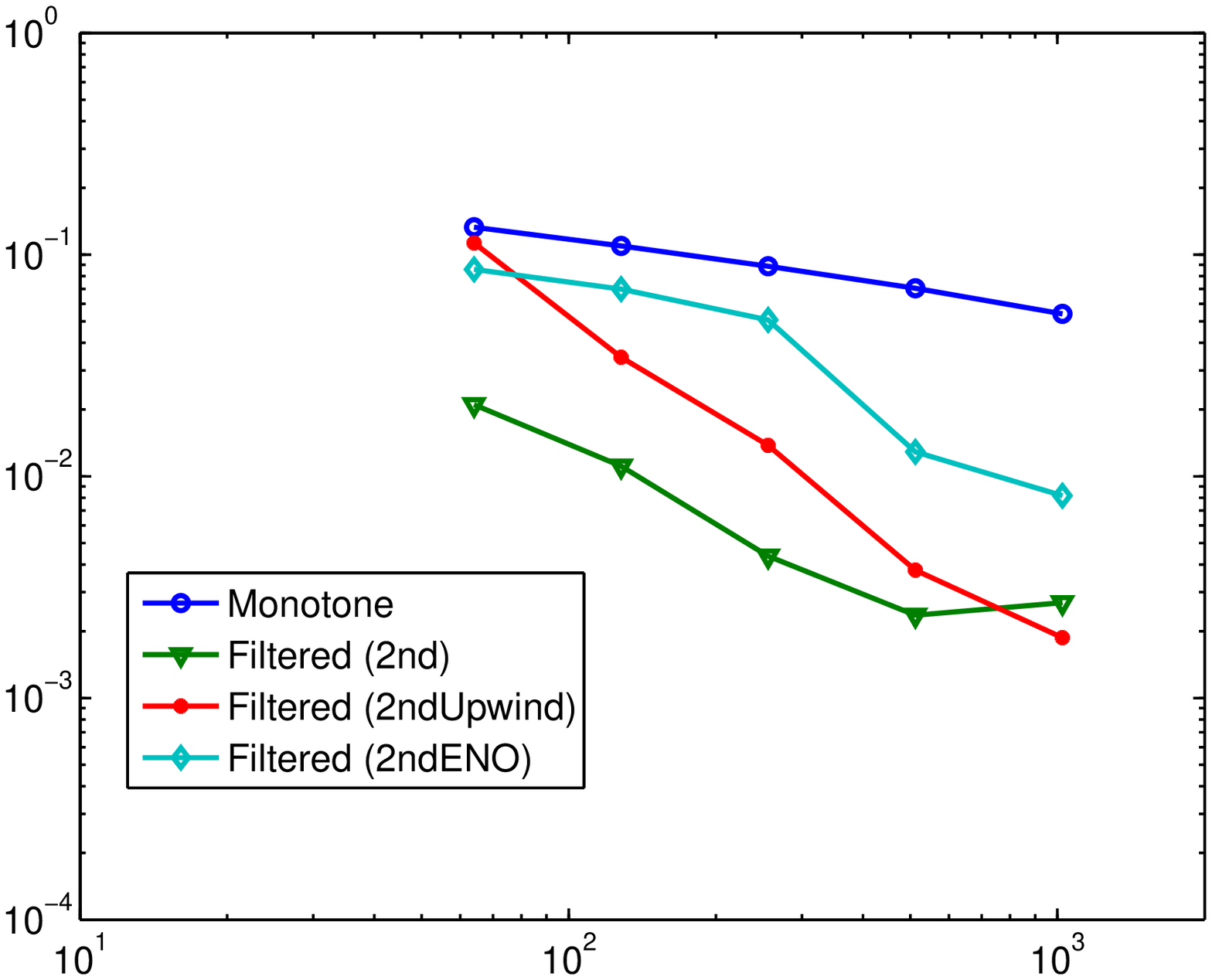}
\end{tabular}
\caption{Log-log plot of the errors for the one-dimensional examples of HJ equations.}
\label{fig:Exloglog1D}
\end{figure}

In the general case of the HJ equations, the results are not as clean. In the first example, the order of convergence remains the same with the monotone scheme still being first order convergent. As for the second example, where the Hamiltonian is not convex, the monotone scheme is not even first order convergent as in all the other examples and we see an increase in the order of convergence for both the second order upwind and ENO filtered schemes. In general we don't expect this increase in the order of convergence of the global accuracy since near the singularity we fall back into the monotone scheme.

\textit{Upwind vs ENO.} The ENO filtered schemes only outperformed the upwind filtered schemes in the first example for the eikonal equation. In this example, both schemes have the same rate of convergence but with ENO schemes having a smaller constant, which can be explained by the fact that the ENO schemes in this example tend to use centered discretizations which have a smaller truncation error than the upwind discretizations. On the other examples, the upwind filtered schemes always performed at least as good as its ENO counterparts.

To finish the discussion, we now take a closer look at the second example for the eikonal equation. In this, the fourth order ENO scheme doesn't have fourth order accuracy and is in fact less accurate than the third order ENO scheme, which also doesn't have third order accuracy. In this case, although never interpolating where the solution is not smooth, the ENO scheme uses three different stencils (see Figure \ref{fig:Ex2ActiveStencils}) which somehow seems to prevent us to get the fourth order accuracy. Moreover, the second order ENO scheme performs an interpolation where the solution is not smooth, although this doesn't affect the rate of convergence of the method (see Figure \ref{fig:Ex2ActiveStencils}). This example illustrates the advantage of using the upwind filtered scheme, which has a fixed stencil, over the ENO scheme, which, while designed heuristically to choose the best stencil, may not always do so. It is worth mentioning that the WENO schemes were introduced to improve the ENO schemes, but these add another layer of complexity without any clear advantage over the filtered upwind schemes.

\subsection{Exact solutions in two dimensions}

\label{sec:examples2D}
In this subsection we discuss the two dimensional examples.
We consider three solutions to the eikonal equation \eqref{eikonal} with $f \equiv 1$, $g \equiv 0$ and $\Gamma$ given by a circle, two points, and a semicircle.
Specifically, we have
\begin{enumerate}
	\item $\Gamma = \left\{(x,y)\in\R^2: x^2+y^2 = 1\right\},$
	\item $\Gamma = \left\{\left(\frac{1}{2},\frac{1}{2}\right),\left(-\frac{1}{2},-\frac{1}{2}\right)\right\},$
	\item $\Gamma = \left\{(x,y)\in\R^2: (x^2+y^2 = 1, x \geq 0) \vee (|y| \leq 1, x = 0)\right\}.$
\end{enumerate}

We chose these examples because the corresponding solutions have varying degrees of regularity. In the first, the solution is smooth (outside $\Gamma$).  In the second we have a singularity  along the line $\left\{(x,y)\in\R^2: x=-y\right\}$ and therefore the solution is only piecewise smooth outside ($\Gamma$). In the third, (outside $\Gamma$) the solution is smooth for $x > 0$ but only Lipschitz continuous for $x <0$. The exact solution is the distance function to the set $\Gamma$.

All computations are performed on the domain $[-2,2]\times[-2,2]$, which is discretized  on an $N\times N$ grid. We assume the exact solution to be known at the neighboring grid points of $\Gamma$ as discussed in subsection \ref{sec:boundary}, except in the second example where we initialize the solution where $u < 0.1$ in order to avoid dealing with the singularities at $\Gamma$ (this is a standard thing to do when studying the higher global accuracy of the methods).

All solutions are displayed in Figure \ref{fig:ExProfiles2D}.
\begin{figure}[htdp]
\centering
\begin{tabular}{ccc}
\hspace{-1in}
\includegraphics[width=0.43\textwidth]{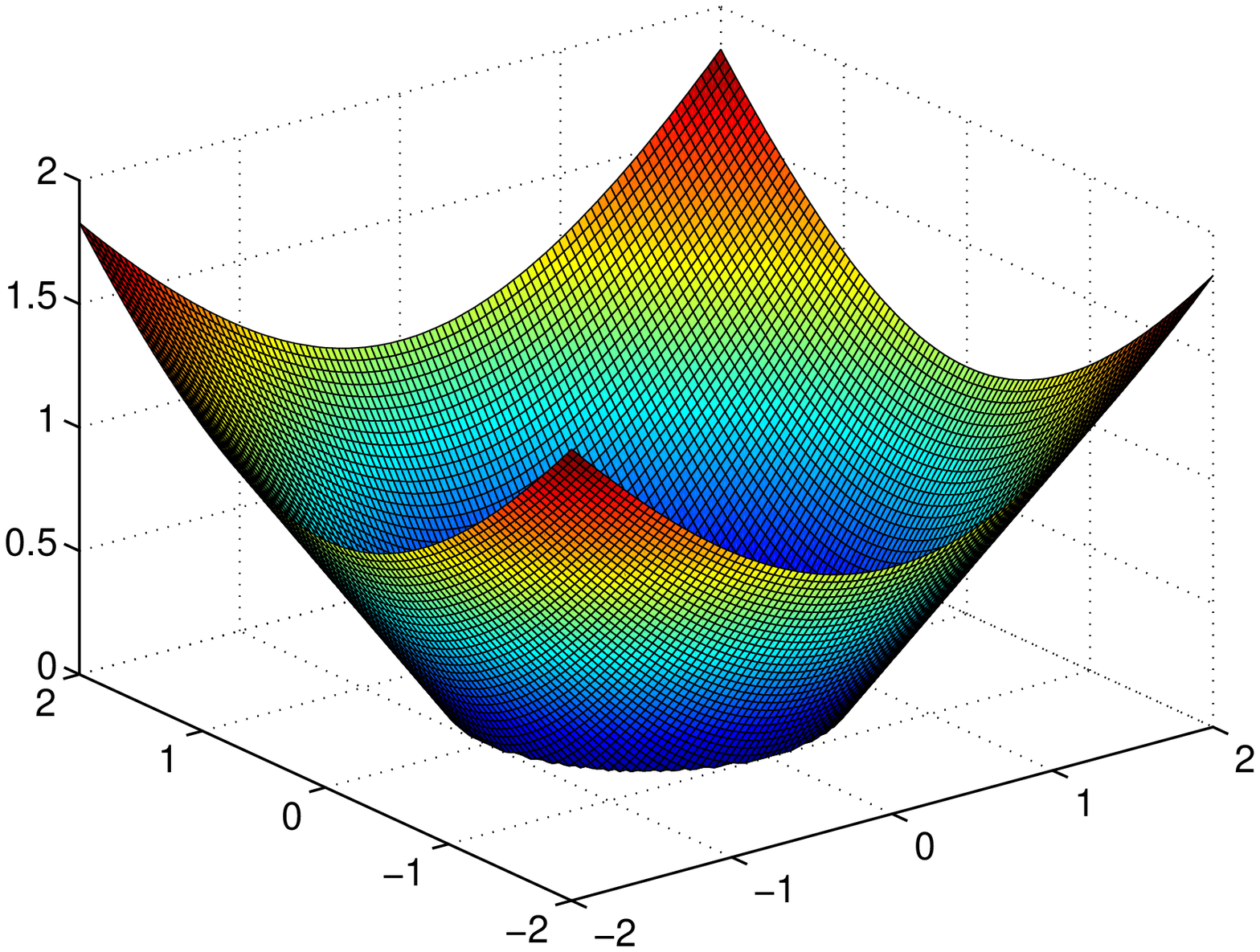} & \includegraphics[width=0.43\textwidth]{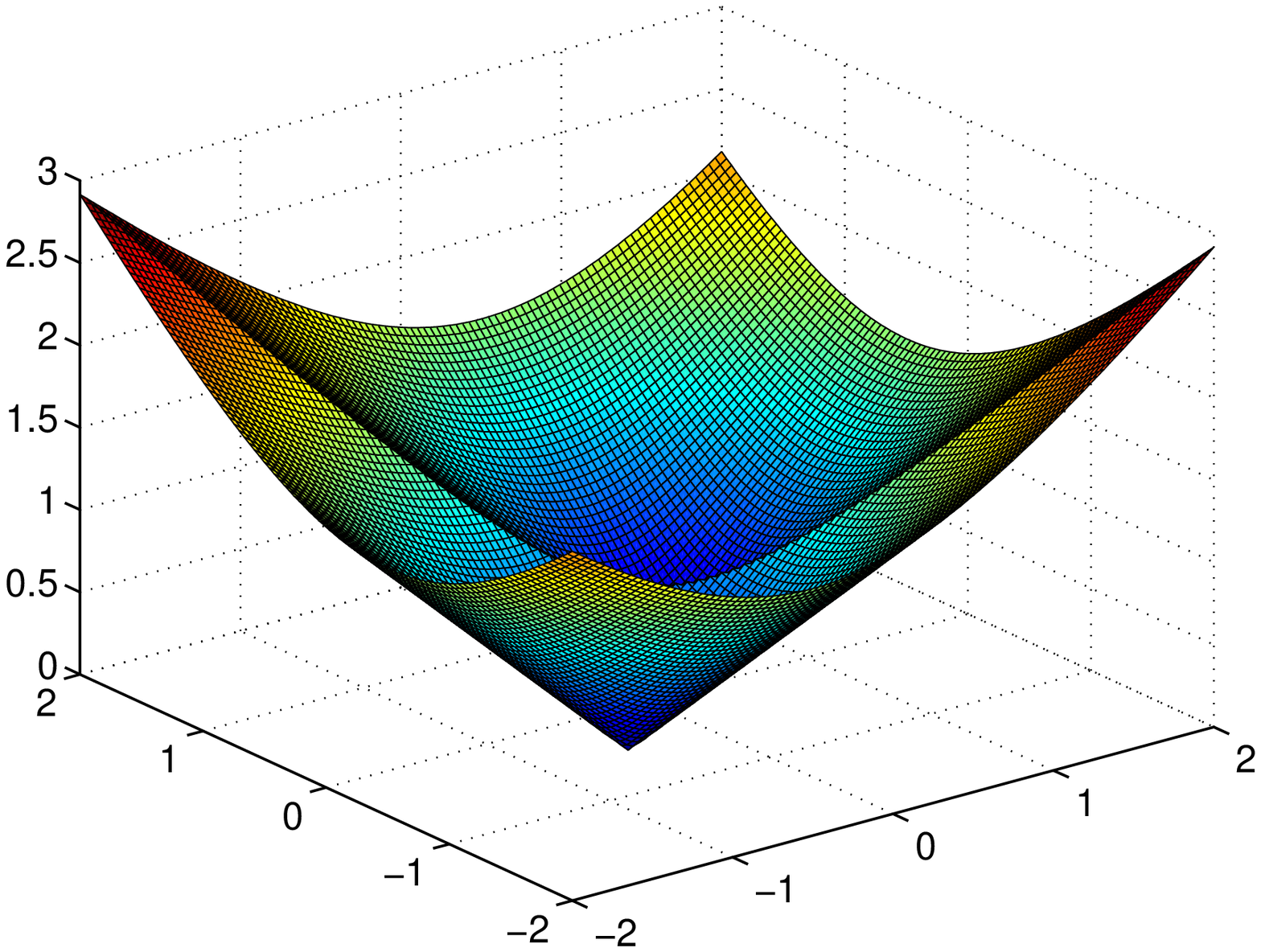} & \includegraphics[width=0.43\textwidth]{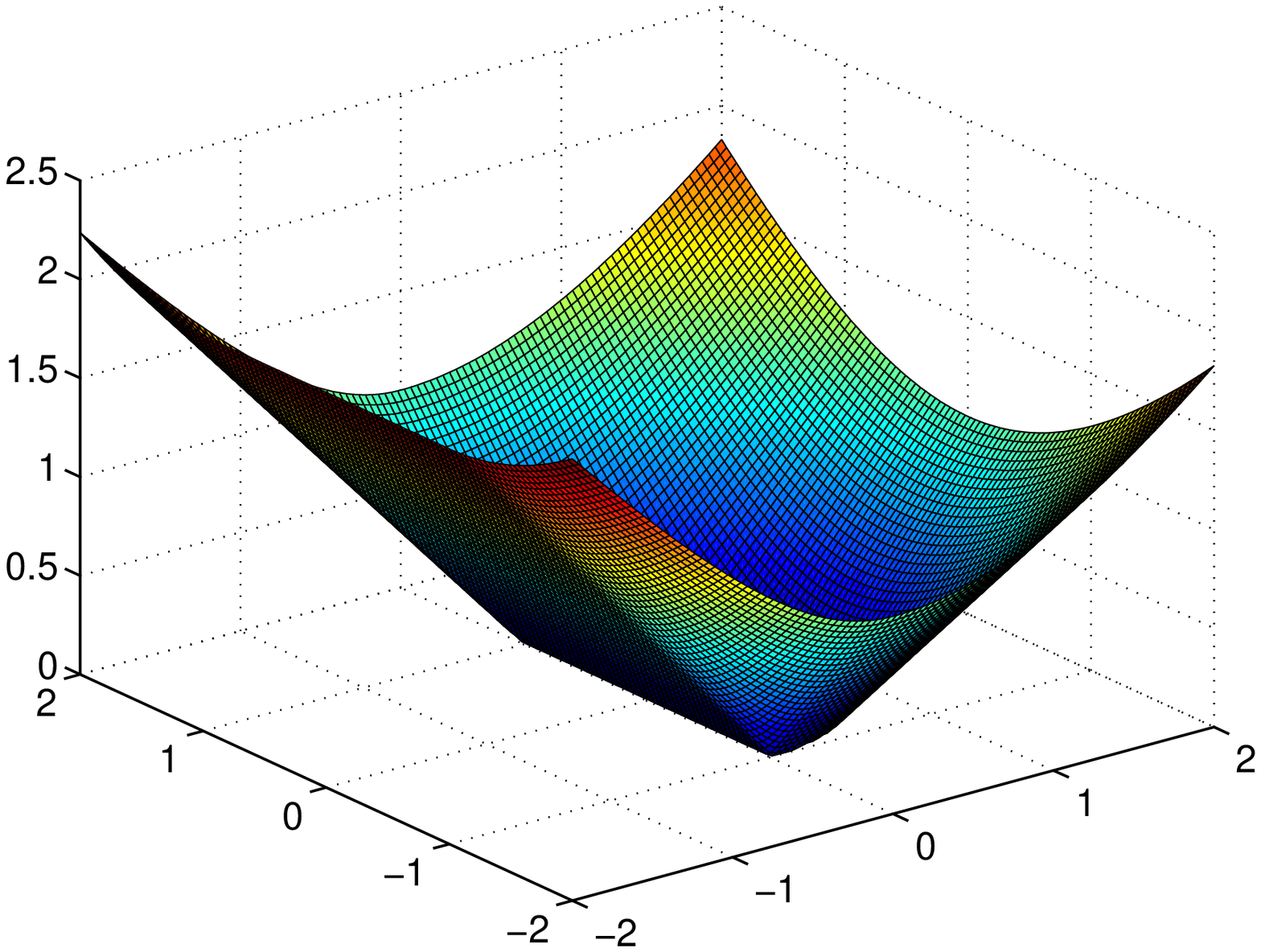}\\
\hspace{-1in}
\includegraphics[width=0.43\textwidth]{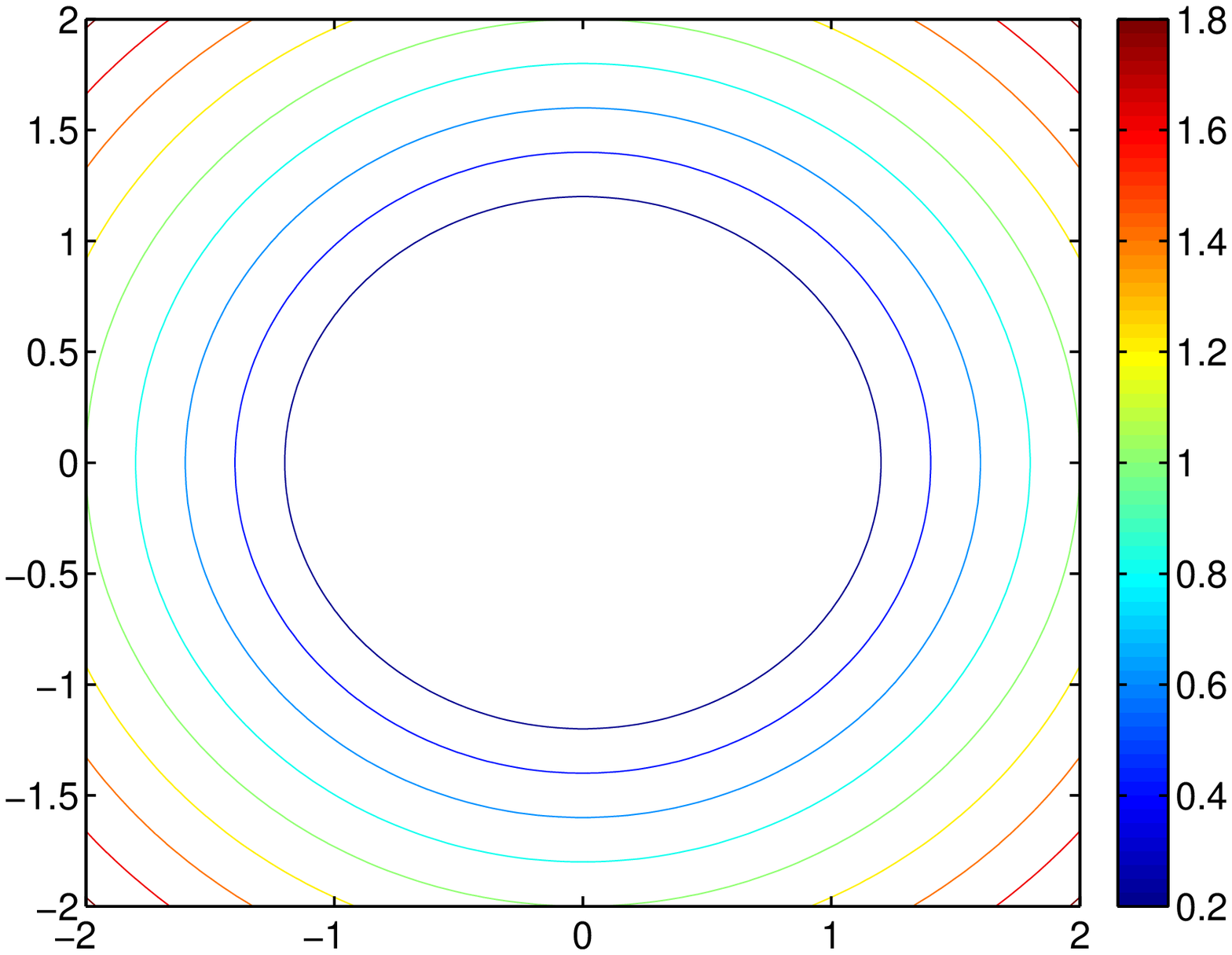} & \includegraphics[width=0.43\textwidth]{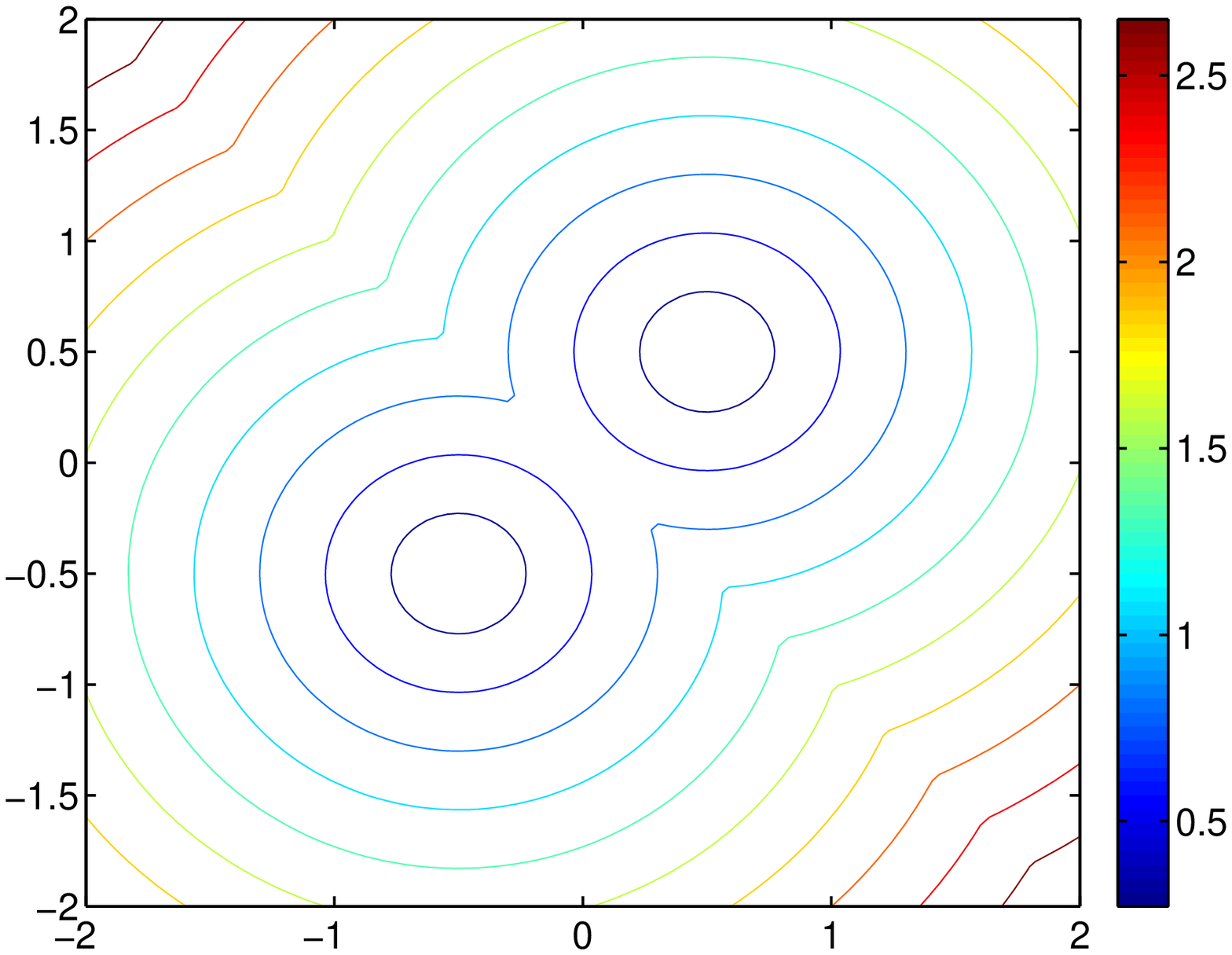} & \includegraphics[width=0.43\textwidth]{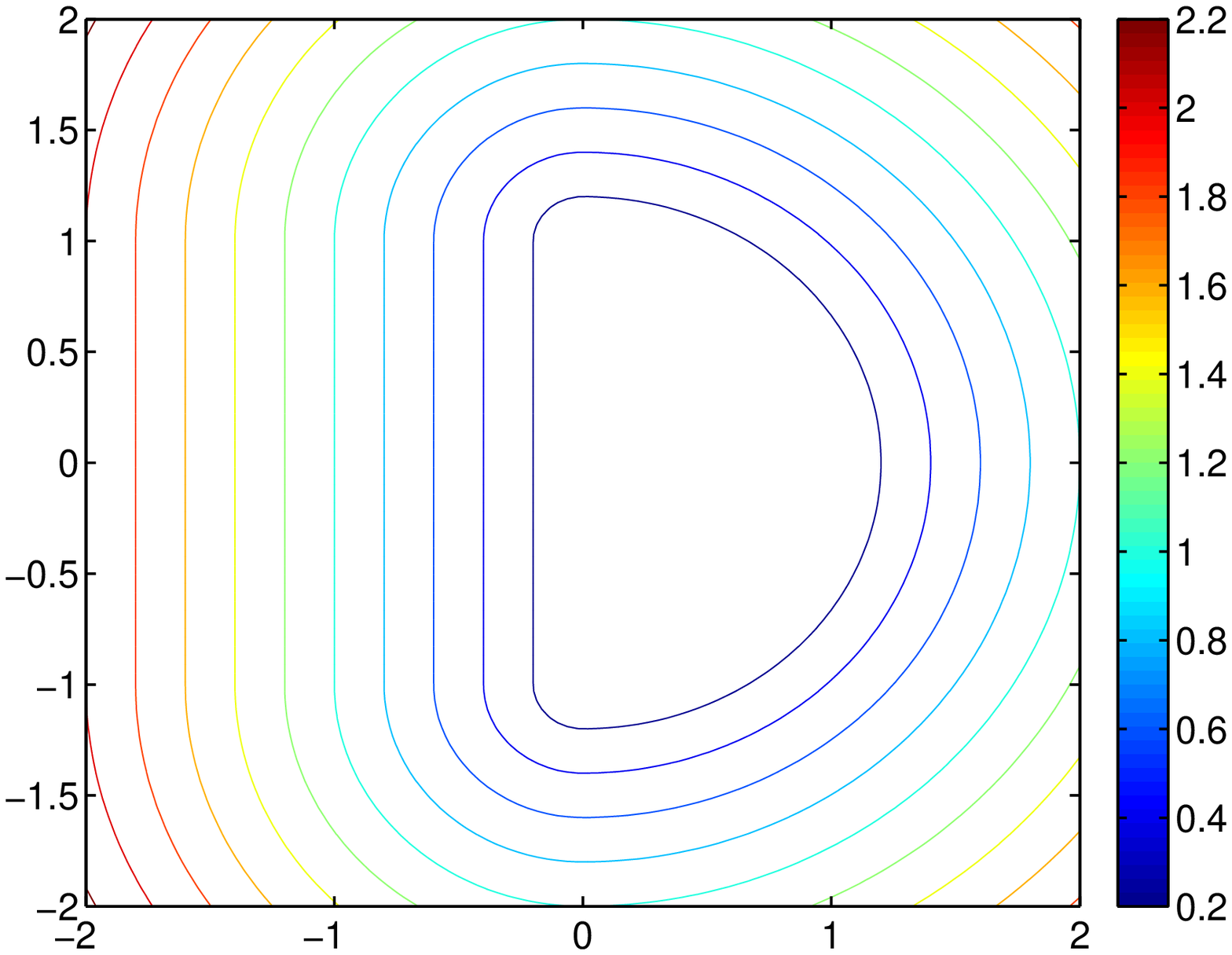}
\end{tabular}
\caption{Profile and contour plots of the solutions of the three examples considered in two dimensions.}
\label{fig:ExProfiles2D}
\end{figure}

\subsection{Computational results in two dimensions}
In this subsection we discuss the computational results obtained in two dimensions. The main purpose is to demonstrate that the filtered scheme achieves the higher order accuracy in the regions where the solution is smooth. We organize the discussion in three parts: accuracy and behavior, order of convergence and upwind vs ENO. We obtained results with the monotone scheme \eqref{monotone2deikonal} and with the respective filtered schemes using as the accurate scheme the second order centered, upwind and ENO schemes.

\begin{table}[htdp]
\centering\footnotesize
\begin{tabular}{ccccccccc}
\multicolumn{9}{c}{Errors and order, $1^{st}$ Example}\\
N & \multicolumn{2}{c}{Monotone} & \multicolumn{2}{c}{$2^{nd}$ centered} & \multicolumn{2}{c}{$2^{nd}$ Upwind} & \multicolumn{2}{c}{$2^{nd}$ ENO}\\
\hline
64 & \num{2.167e-02} & - & \num{9.034e-04} & - & \num{1.107e-03} & - & \num{5.284e-04} & - \\
128 & \num{1.126e-02} & 0.93 & \num{2.368e-04} & 1.91 & \num{3.030e-04} & 1.85 & \num{1.476e-04} & 1.82 \\
256 & \num{5.661e-03} & 0.99 & \num{5.964e-05} & 1.98 & \num{7.627e-05} & 1.98 & \num{3.766e-05} & 1.96 \\
512 & \num{2.854e-03} & 0.99 & \num{1.516e-05} & 1.97 & \num{1.949e-05} & 1.96 & \num{9.682e-06} & 1.95 \\
1024 & \num{1.432e-03} & 0.99 & \num{3.893e-06} & 1.96 & \num{4.903e-06} & 1.99 & \num{2.444e-06} & 1.98 \\
\end{tabular}
\caption{Accuracy and order of convergence of the schemes for the first example in two dimensions in the $l^\infty$ norm.}
\label{table:Ex1errors2D}
\end{table}

\begin{table}[htdp]
\centering\footnotesize
\begin{tabular}{ccccccccc}
\multicolumn{9}{c}{Errors and order, $2^{nd}$ Example}\\
N & \multicolumn{2}{c}{Monotone} & \multicolumn{2}{c}{$2^{nd}$ centered} & \multicolumn{2}{c}{$2^{nd}$ Upwind} & \multicolumn{2}{c}{$2^{nd}$ ENO}\\
\hline
64 & \num{5.128e-02} & - & \num{1.643e-02} & - & \num{1.276e-02} & - & \num{1.297e-02} & - \\
128 & \num{2.663e-02} & 0.93 & \num{1.016e-02} & 0.69 & \num{9.837e-03} & 0.37 & \num{9.514e-03} & 0.44 \\
256 & \num{1.326e-02} & 1.00 & \num{5.485e-03} & 0.88 & \num{5.121e-03} & 0.94 & \num{4.795e-03} & 0.98 \\
512 & \num{6.640e-03} & 1.00 & \num{3.019e-03} & 0.86 & \num{2.600e-03} & 0.98 & \num{2.402e-03} & 0.99 \\
1024 & \num{3.324e-03} & 1.00 & \num{1.483e-03} & 1.02 & \num{1.425e-03} & 0.87 & \num{1.490e-03} & 0.69 \\
\hline\end{tabular}
\caption{Accuracy and order of convergence  of the schemes for the second example in two dimensions in the $l^\infty$ norm.}
\label{table:Ex2errors2D}
\end{table}

\begin{table}[htdp]
\centering\footnotesize
\begin{tabular}{ccccccccc}
\multicolumn{9}{c}{Errors and order, $2^{nd}$ Example}\\
N & \multicolumn{2}{c}{Monotone} & \multicolumn{2}{c}{$2^{nd}$ centered} & \multicolumn{2}{c}{$2^{nd}$ Upwind} & \multicolumn{2}{c}{$2^{nd}$ ENO}\\
\hline
64 & \num{4.310e-01} & - & \num{4.355e-02} & - & \num{7.111e-02} & - & \num{3.589e-02} & - \\
128 & \num{2.202e-01} & 0.96 & \num{1.331e-02} & 1.69 & \num{1.967e-02} & 1.83 & \num{1.002e-02} & 1.82 \\
256 & \num{1.088e-01} & 1.01 & \num{2.893e-03} & 2.19 & \num{4.844e-03} & 2.01 & \num{2.538e-03} & 1.97 \\
512 & \num{5.420e-02} & 1.00 & \num{9.942e-04} & 1.54 & \num{1.233e-03} & 1.97 & \num{6.506e-04} & 1.96 \\
1024 & \num{2.706e-02} & 1.00 & \num{2.697e-04} & 1.88 & \num{3.149e-04} & 1.97 & \num{1.697e-04} & 1.94 \\
\hline\end{tabular}
\caption{Accuracy and order of convergence  of the schemes for the second example in two dimensions in the $l^1$ norm.}
\label{table:Ex2errors2DL1}
\end{table}

\begin{table}[htdp]
\centering\footnotesize
\begin{tabular}{ccccccccc}
\multicolumn{9}{c}{Errors and order, $2^{nd}$ Example}\\
N & \multicolumn{2}{c}{Monotone} & \multicolumn{2}{c}{$2^{nd}$ centered} & \multicolumn{2}{c}{$2^{nd}$ Upwind} & \multicolumn{2}{c}{$2^{nd}$ ENO}\\
\hline
64 & \num{5.771e-02} & - & \num{9.083e-03} & - & \num{9.342e-03} & - & \num{8.811e-03} & - \\
128 & \num{3.541e-02} & 0.70 & \num{4.833e-03} & 0.90 & \num{5.508e-03} & 0.75 & \num{4.566e-03} & 0.94 \\
256 & \num{2.117e-02} & 0.74 & \num{2.399e-03} & 1.00 & \num{3.344e-03} & 0.72 & \num{2.605e-03} & 0.81 \\
512 & \num{1.238e-02} & 0.77 & \num{1.470e-03} & 0.70 & \num{2.523e-03} & 0.41 & \num{1.574e-03} & 0.72 \\
1024 & \num{7.112e-03} & 0.80 & \num{1.024e-03} & 0.52 & \num{1.517e-03} & 0.73 & \num{1.055e-03} & 0.58 \\
\hline\end{tabular}
\caption{Accuracy and order of convergence of the schemes for the third example in two dimensions in the $l^\infty$ norm.}
\label{table:Ex3errors2D}
\end{table}

\textit{Accuracy and behavior of the filtered schemes}
We begin with the results presented in Figure \ref{fig:Exloglog2D} and Tables \ref{table:Ex1errors2D}, \ref{table:Ex2errors2D}, \ref{table:Ex3errors2D}. It is clear the solutions computed using the filtered schemes are more accurate.

The behavior of the filtered schemes is very much like the one obtained in the one-dimensional examples: in first example, the monotone scheme is never used since the solution is smooth; in the second example, it's only used near the singularity at $x=-y$; in the third example, it's only used near the corners of $\Gamma$.

\textit{Order of convergence}
Unlike the one dimensional case for the eikonal equation, the rate of convergence of the error in the $l^\infty$ norm can be less than the formal order of accuracy of the accurate schemes and will depend on the smoothness of the solutions. In the first example, the solution is smooth and we obtain second order convergence in the $l^\infty$ norm (see Figure \ref{fig:Exloglog2D} and Table \ref{table:Ex1errors2D}). This was expected since the ``equivalent'' fast marching method was already proven second order convergent for smooth solutions in \cite{Ahmed}. In the second example, we have a shock of co-dimension $1$ and therefore we get first order rate convergence in the $l^\infty$ norm and second order in the $l^1$ norm (see Figures \ref{fig:Exloglog2D}, \ref{fig:Exloglog2DL1} and Tables \ref{table:Ex2errors2D}, \ref{table:Ex2errors2DL1}). We can still see the second order of convergence in the $l^\infty$ norm if we look away from the singularities (see Figure \ref{fig:Ex2loglog2DRegions}). As for the third example, we do not have shocks, but the solution is still not smooth due to the corners in $\Gamma$ which have a rarefaction effect much like the ones in hyperbolic conversation laws. For instance, in the region $\left\{(x,y)\in\R^2:x<0, y > 1\right\}$ all characteristics emanate from the point $(0,1)$ and so the errors incurred there will propagate out and pollute the solution. Thus the error is globally first order in both the $l^\infty$ and $l^1$ norm. However if we restrict the errors to the region $\left\{(x,y)\in\R^2:x^2+y^2 \geq 1, x \geq 0.1\right\}$ where the solution is smooth we do obtain second order rate of convergence in the $l^\infty$ norm (see Figure \ref{fig:Ex3loglog2DRegions}). Finally, in region $\left\{(x,y)\in\R^2: |y| \leq 0.8, x \leq 0\right\}$, all the schemes were exact up to machine precision since they are exact on flat regions.

\begin{figure}[htdp]
\centering
\begin{tabular}{c}
\includegraphics[width=0.65\textwidth]{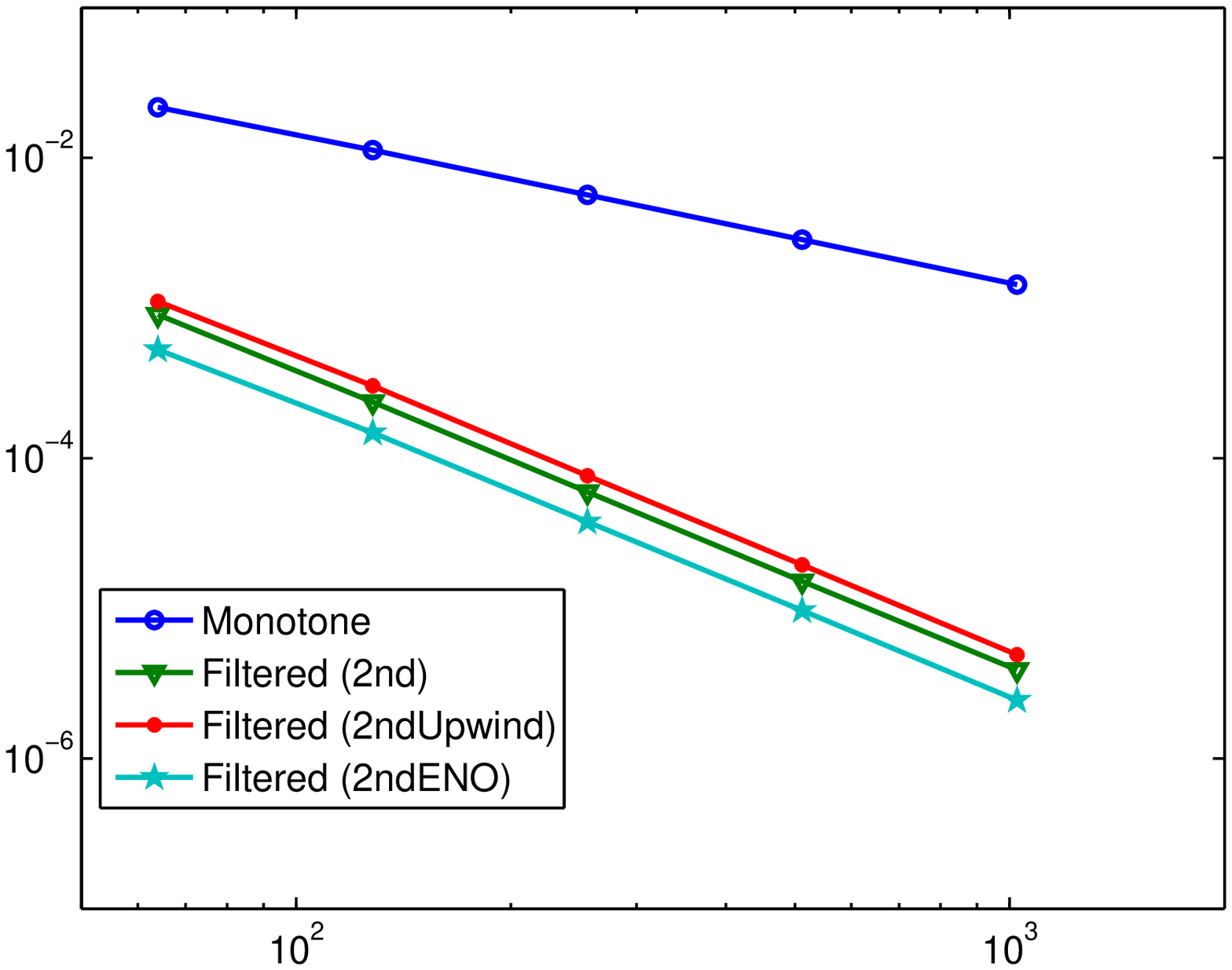}\\
\includegraphics[width=0.65\textwidth]{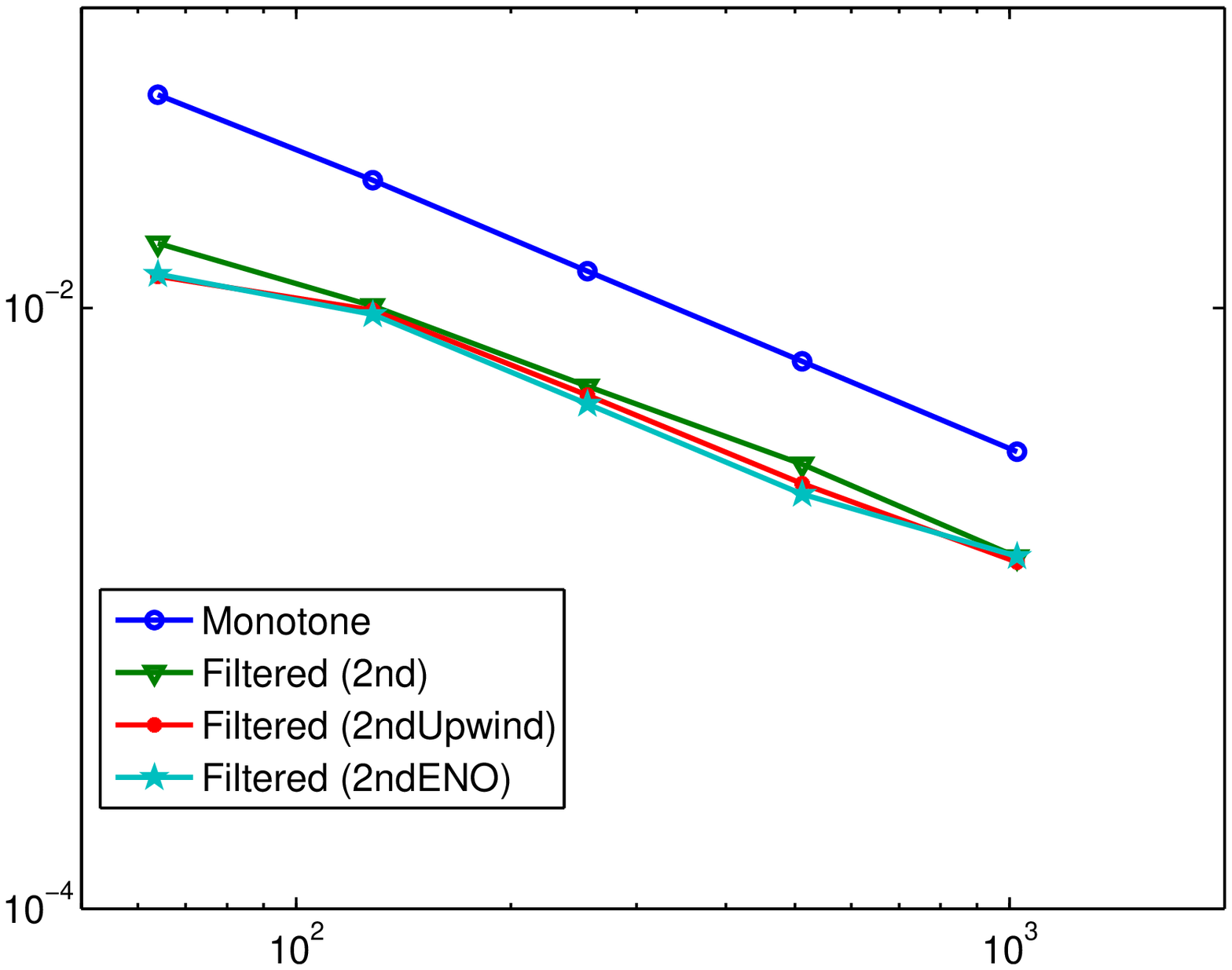}\\
\includegraphics[width=0.65\textwidth]{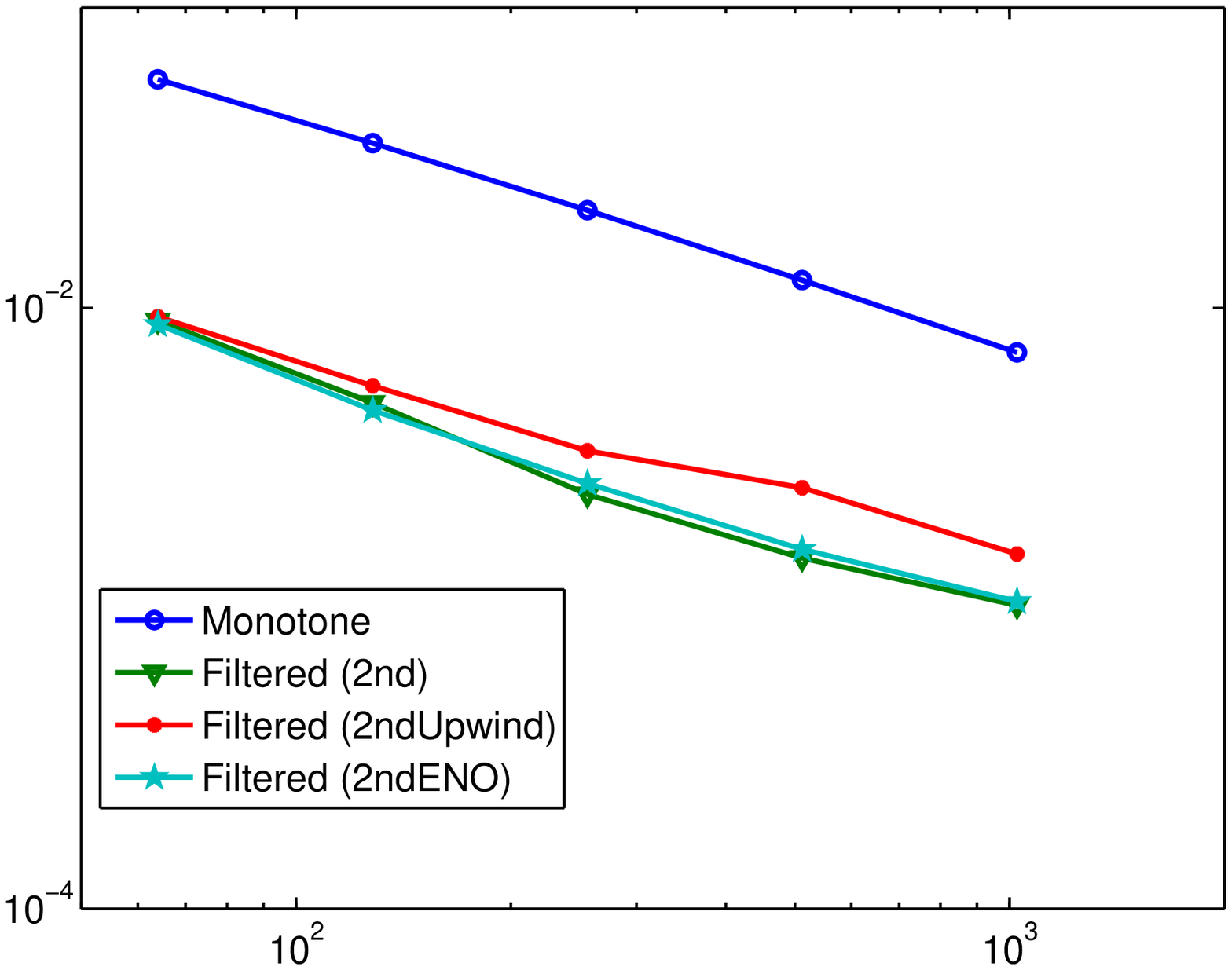}
\end{tabular}
\caption{Log-log plot of the errors for the two-dimensional examples in the $l^\infty$ norm.}
\label{fig:Exloglog2D}
\end{figure}

\begin{figure}[htdp]
\centering
\begin{tabular}{c}
\includegraphics[width=0.65\textwidth]{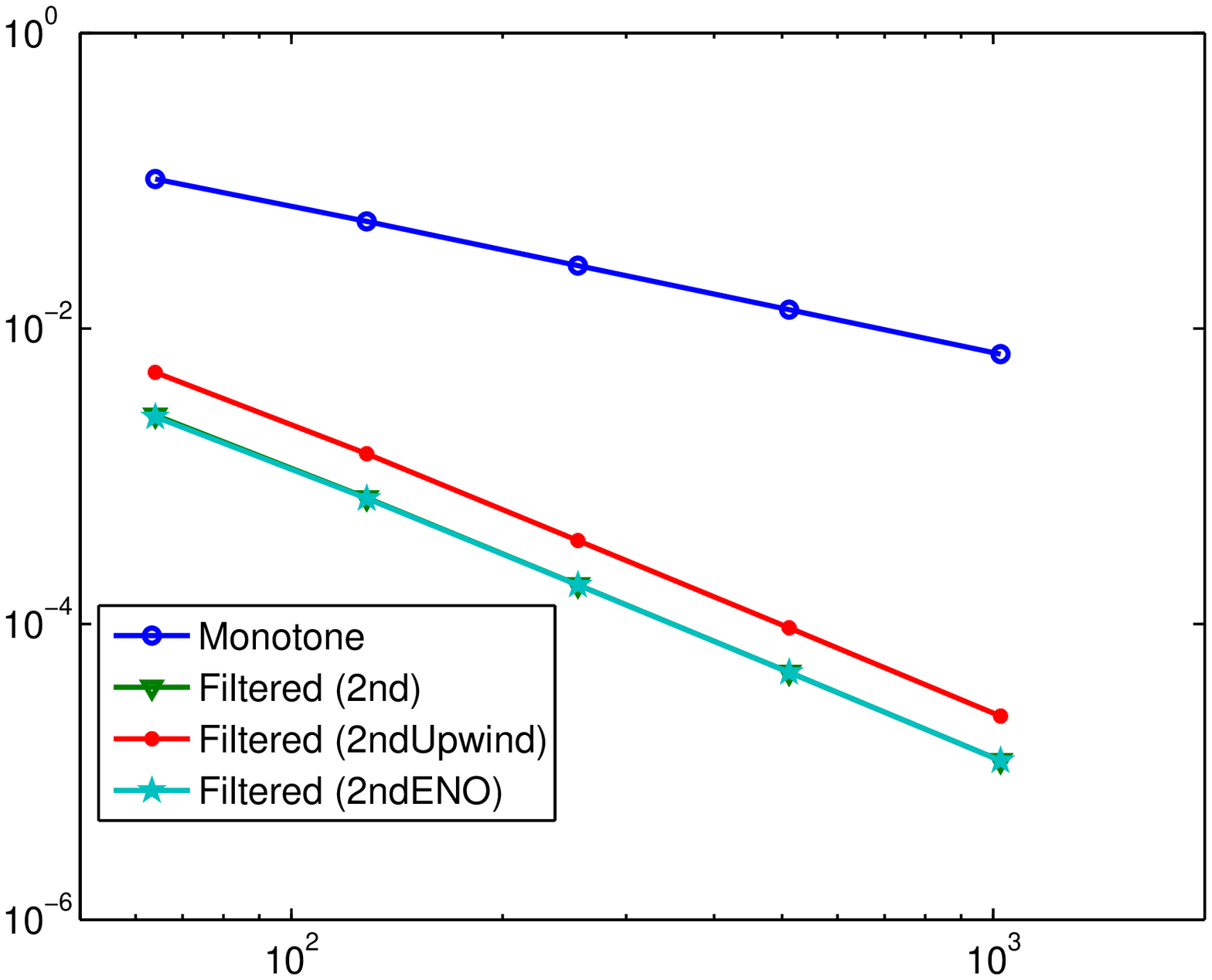}\\
\includegraphics[width=0.65\textwidth]{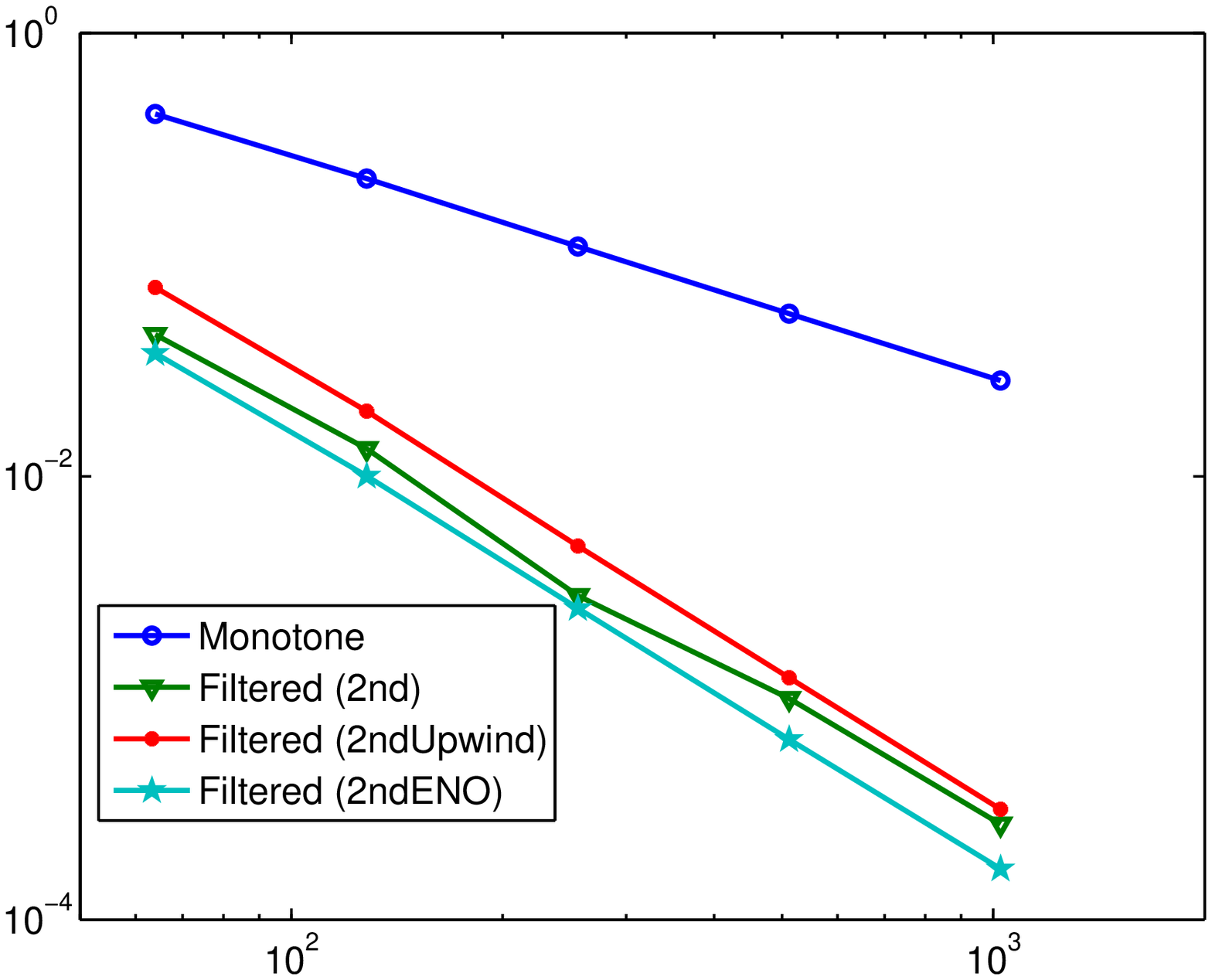}\\
\includegraphics[width=0.65\textwidth]{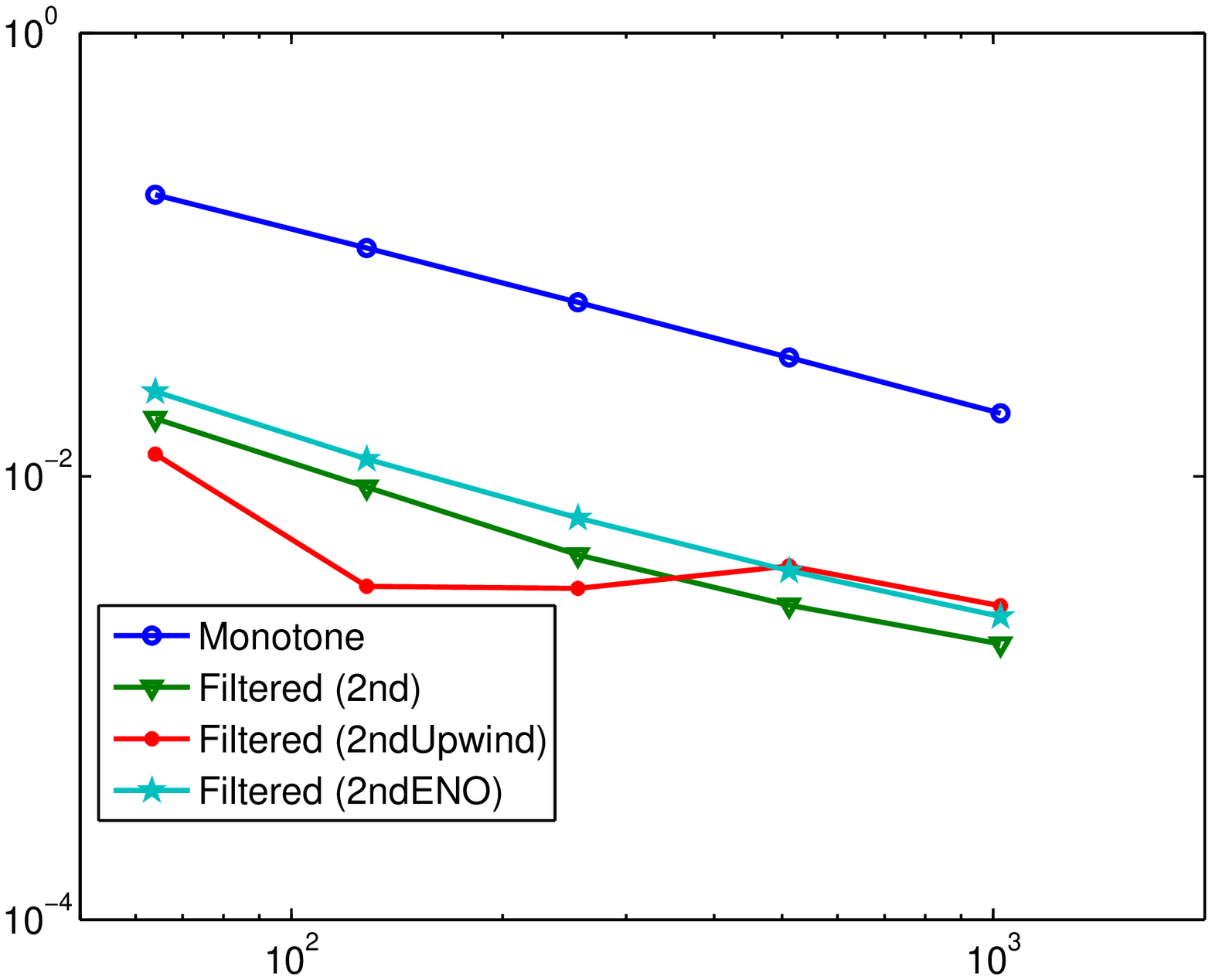}
\end{tabular}
\caption{Log-log plot of the errors for the two-dimensional examples in the $l^1$ norm.}
\label{fig:Exloglog2DL1}
\end{figure}

\begin{figure}[htdp]
\centering
\begin{tabular}{cc}
\hspace{-1in}
\includegraphics[width=0.65\textwidth]{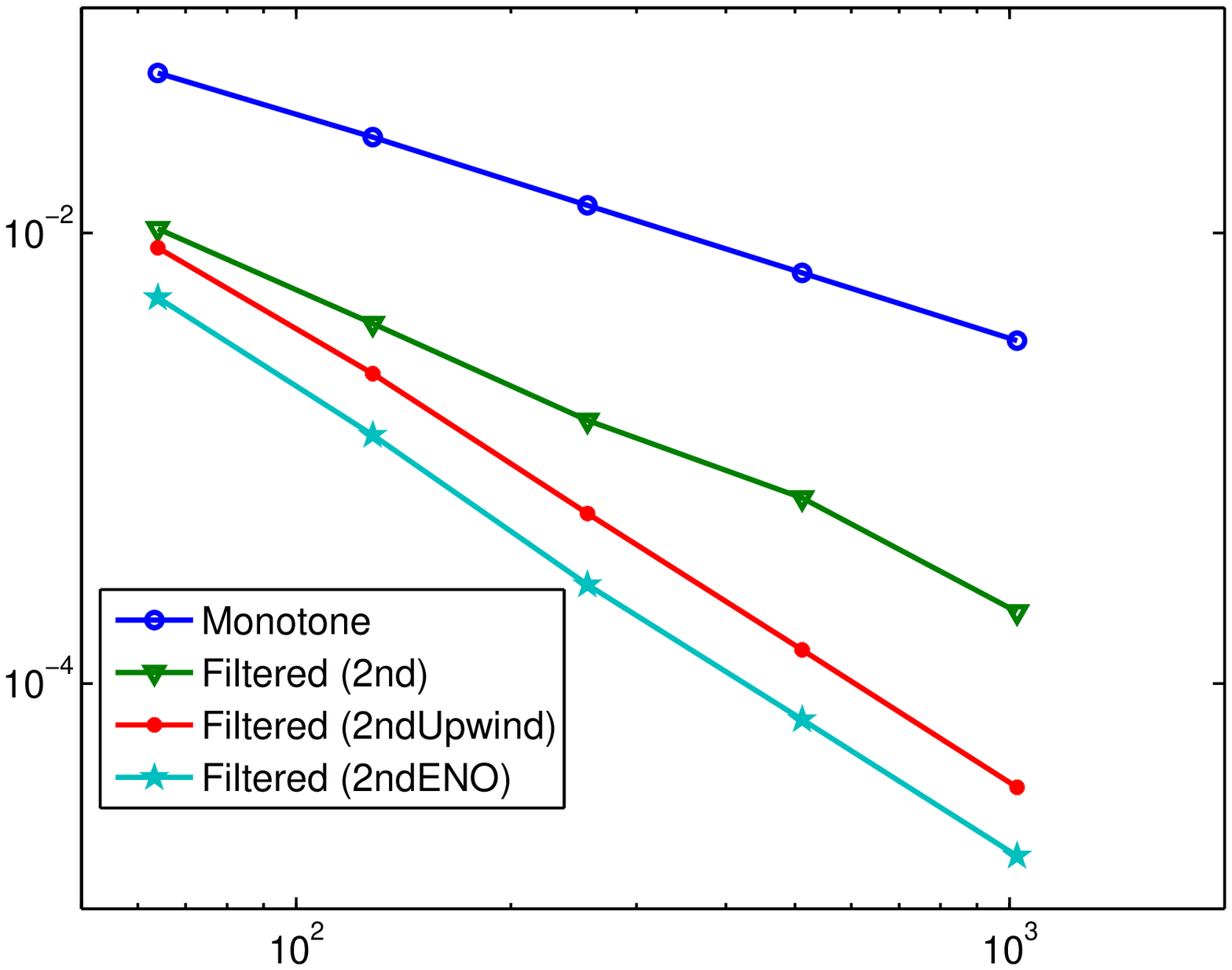} & \includegraphics[width=0.65\textwidth]{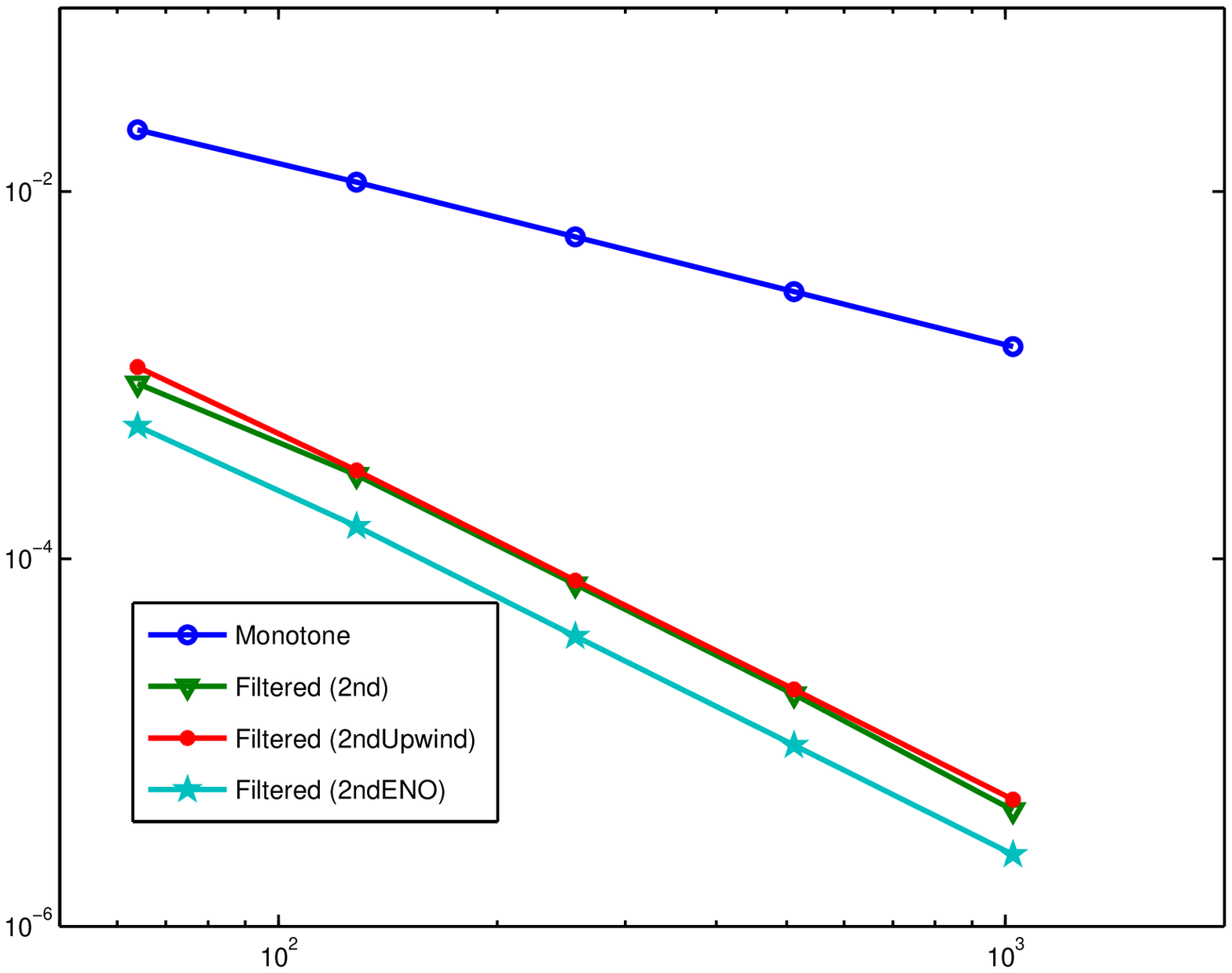}
\end{tabular}
\caption{Log-log plot of the errors for: (a) the second example in the $l^\infty$ norm in regions $\left\{(x,y)\in\R^2: |x+y|>0.1\right\}$; (b) the third example in the $l^\infty$ norm in regions $\left\{(x,y)\in\R^2:x^2+y^2 \geq 1, x \geq 0.1\right\}$.}
\label{fig:Exloglog1D}
\end{figure}

\subsubsection*{Upwind vs ENO}
Comparing the upwind schemes to the ENO schemes, we see that we obtained similar results with the difference being a smaller asymptotical constant. This is explained by the fact that ENO schemes tend to use centered discretizations which have a smaller truncation error than the upwind discretizations.

Third order upwind and ENO filtered schemes were also used, but they didn't show any advantage over the second order schemes. We didn't even obtain the third order rate of convergence for the first example even though the solution is smooth. This is most likely related to a result proven in \cite{Ahmed}. There the authors show that the ``equivalent'' third order fast marching method is unstable. They also provide an alternative scheme which uses full two-dimensional stencils and that it's provable third order globally convergent in the $l^\infty$ norm for smooth solutions. We expect that if we use that scheme as our accurate scheme we would obtain a filtered scheme with the same order of convergence.
\section{Conclusions}

We introduce filtered schemes for Hamilton-Jacobi equations, which allow us to construct convergent, high order accurate finite difference schemes. These schemes are extremely flexible in the choice of accurate scheme, and so they allow for a wide range of existing discretizations (even unstable ones) to be used, while retaining the stability and convergence proof of the monotone schemes.

Focusing on the special, but important case of the eikonal equation, we tested the accuracy of several discretizations on solutions of varying regularity in one and two dimensions. In one dimension, we used filtered central differences, filtered higher order upwinding, and filtered ENO schemes.  In each case we obtained higher accuracy, even in regions where the solution was not smooth. For the eikonal equation case we were able to prove the higher convergence rate. This result, although very special to the eikonal equation, illustrates the potential accuracy of the method.

Due to the explicit nature of the filtered upwind schemes we were able to use the simple but effective fast sweeping method to compute solutions. In the case of filtered ENO, a slower iterative method was used. We also gave a comparison using filtered ENO schemes, and found an example where the error for ENO was greater than its formal accuracy.

The convergence results in two dimensions were more complicated, but more generic, in that we expect similar results on more general HJ equations.   In this case, for smooth solutions, we obtained second order accuracy. The same order of accuracy has been previously obtained by several authors using more complicated schemes as opposed to the simplicity of the upwind filtered schemes. In particular, our filtered upwind schemes in two dimensions are still explicit, thus allowing the use of the fast sweeping method to obtain solutions.

The schemes developed here are simple to implement, and allow an unrestricted choice of higher order discretizations to be used.  While we mainly focused on a particular type equation (HJ equations), it should be clear that the filtered schemes can be used in much wider context, while still retaining the advantages of accuracy, stability and convergence to the viscosity solution of the monotone schemes.

\appendix

\section{Convergence proof of the filtered schemes}\label{appendix}

In this Appendix we give the proof to Theorem \ref{thm:converge}, which we will restate. The Appendix is organized in three parts: in the first one, we briefly discuss viscosity solutions for Hamilton-Jacobi equations; in the second, we recall the definitions of consistency, accuracy, monotonicity and stability for approximations schemes; in the third and last part, we give the convergence proof.

\subsection{Viscosity Solutions}\label{appx:viscosity}

We are interested in solving
\bq\label{HJequationAppendix}\begin{cases}
H(x,\nabla u) = f(x),	& x \in \Omega,\\
u(x) = g(x),			& x \in \Gamma,
\end{cases}
\eq
where $\nabla u$ is the gradient of the function $u$, $\Omega$ is an open set, $\Gamma$ is the boundary of $\Omega$ and the Hamiltonian $H$ is a nonlinear Lipschitz continuous function.

We introduce the function $F: \Omega \times \R \times \Rd \to \R$ which we define as
\[
F(x,r,p) = \begin{cases}
H(x,p)-f(x)	& x \in \Omega,\\
r-g(x)	& x \in \Gamma.
\end{cases}\]
Then $u \in C^1(\Omega)$ is a solution of \eqref{HJequationAppendix} if
\bq\label{PDE}\tag{PDE}
F[u](x) = F(x,u(x),\nabla u(x)) = 0, \quad x \in \Omega.
\eq
However, we won't always have classical solutions which motivates the definition of viscosity solutions, a weak form of solutions. Before we define it, we introduce the upper and lower semicontinuous envelopes of a function.
\begin{definition}[upper and lower semicontinuous envelopes]
The upper and lower semicontinuous envelopes of a function $u$ are defined, respectively, by
\begin{align*}
u^*(x) & = \limsup_{y \to x}u(y),\\
u_*(x) & = \liminf_{y \to x}u(y).
\end{align*}
\end{definition}

It's easy to see that for $F_*$ and $F^*$ we have
\[\begin{cases}
F_*(x,r,p) = F^*(x,r,p) = H(x,p)-f(x)	& x \in \Omega,\\
F_*(x,r,p) = \min\{H(x,p),r-g(x)\}		& x \in \Gamma,\\
F^*(x,r,p) = \max\{H(x,p),r-g(x)\}		& x \in \Gamma.\\
\end{cases}\]

\begin{definition}[viscosity solution]
An upper (lower) semicontinuous function $u$ is a viscosity subsolution (supersolution) of \eqref{PDE} if for every $\phi \in C^1\left(\overline{\Omega}\right)$, whenever $u-\phi$ has a local maximum (minimum) at $x \in \overline{\Omega}$, then $F_*(x,u(x),\nabla u(x)) \leq 0$ $\left(F^*(x,u(x),\nabla u(x)) \geq 0\right)$. A function $u$ is a viscosity solution if it both a subsolution and supersolution.
\end{definition}

\begin{remark}
When checking the definition of a viscosity solution we can limit ourselves to considering unique, strict, global maxima (minima) of $u-\phi$ with a value of zero at the extremum. See, for exemple, \cite[Prop 2.2]{KoikeViscosity}.
\end{remark}

We assume that \eqref{PDE} satisfies a comparison principle: if $u \in \USC\left(\overline{\Omega}\right)$ is a subsolution and $v \in \LSC\left(\overline{\Omega}\right)$ is a supersolution of \eqref{PDE}, then $u \leq v$ on $\overline{\Omega}$. The proof of this result is one of the main technical arguments in the viscosity solutions theory \cite{CIL}.

\subsection{Approximation Schemes}

An approximation scheme is a family of functions parameterized by $h \in \R^+$
\[F^h:\overline{\Omega}\times \R \times L^\infty\left(\overline{\Omega}\right)\to \R\]
which we write as $F^\delta(x,r,u(\cdot))$. Given a function $u \in L^\infty\left(\overline{\Omega}\right)$, we write
\begin{align}\label{PDEe}  \tag*{$(\textnormal{PDE})^h$}
F^h[u](x) = F^h(x,u(x),u(\cdot)).
\end{align}
The function $u^h$ is a solution of the scheme $F^h$ if
\[F^h[u^h](x) = 0, \quad \text{ for all } x \in \overline{\Omega}.\]
In general, the approximation schemes come from finite difference schemes (as they do here in this paper): $h$ is the grid size and the function on the grid is continuously extended to the domain by using interpolation.

We now introduce some important properties for these schemes which guarantee their convergence in a more general setting than in \cite{BSnum}.

\begin{definition}[consistent]
The scheme \ref{PDEe} is consistent with the equation \eqref{PDE} if for any smooth function $\phi$ and $x \in \overline{\Omega}$
\[\limsup_{h\to 0, y\to x,\xi\to 0} F^h(y,\phi(y)+\xi,\phi(\cdot)+\xi) \leq F^*(x,\phi(x),\nabla\phi(x))\]
\[\liminf_{h\to 0, y\to x,\xi\to 0} F^h(y,\phi(y)+\xi,\phi(\cdot)+\xi) \geq F_*(x,\phi(x),\nabla\phi(x)).\]
\end{definition}

\begin{definition}[accurate]
The scheme \ref{PDEe} is $\alpha$-order accurate if for any smooth function $\phi$ and $x \in \Omega$
\[F^h[\phi](x)-F[\phi](x) = \mathcal{O}(h^\alpha).\]
\end{definition}

\begin{remark}
We define accuracy only inside the domain.
\end{remark}

\begin{definition}[stable]
The scheme \ref{PDEe} is stable if any solution $u^h$ of \ref{PDEe} is bounded independently of $h$.
\end{definition}

\begin{definition}[monotone]\label{monotoneAppendix}
The scheme \ref{PDEe} is monotone if for every $h > 0$, $x \in \overline{\Omega}$, $s \in \R$ and $u,v \in L^\infty\left(\overline{\Omega}\right)$,
\[u \geq v \Longrightarrow F^h(x,s,u(\cdot)) \leq F^h(x,s,v(\cdot)).\]
\end{definition}

We recall that in this paper we consider $F^h$ to be the filtered scheme given by
\bq\label{defnFilteredAppendix}
F^h[u] = \begin{cases}
F^h_A[u], & \text{ if } \left|F^h_A[u]-F^h_M[u]\right| \leq \sqrt{h},\\
F^h_M[u],   & \text{otherwise}
\end{cases}
\eq
where we take $F_M^h$ to be a consistent monotone scheme and $F^h_A$ a consistent accurate scheme.

\subsection{Convergence Proof}\label{appx:proof}

We can now give the proof of Theorem \ref{thm:converge}, which we restate.

\begin{theorem}[Convergence of Approximation Schemes]
Let $u$ be the unique viscosity solution of \eqref{HJequation}.
For each $h>0$, let $u^h$ be a stable solution of \ref{PDEe}, where the filtered scheme $F^h$ is given by ~\eqref{defnFilteredAppendix} and $F^h_M$ is a consistent and monotone scheme. Then 
\[
u^h \to u, \quad \text{ locally uniformly,  as } h \to 0.
\]
\end{theorem}

\begin{proof}
Define
\begin{align*}
\usub	& = \lim_{h\to 0}\sup_{y\to x} u(y) \in \USC(\overline{\Omega})\\
\usup	& = \lim_{h\to 0}\inf_{y\to x} u(y) \in \LSC(\overline{\Omega})
\end{align*}
From the stability of the solutions $u^h$, it follows that both $\usub$ and $\usup$ are bounded. In addition, we know that $\usup \leq \usub$.

Assume for now that $\usub$ is a subsolution and $\usup$ is a supersolution. Then from the comparison principle for \eqref{PDE} applied to $\usub$ and $\usup$, we conclude that $\usub \leq \usup$. We can then conclude that $\usup = \usub$ and therefore $u$ is the unique solution of \eqref{PDE}, again by the comparison principle for \eqref{PDE}. The local uniform convergence follows from the definitions of $\usub$ and $\usup$.

It then remains to show the claim that $\usub$ is a subsolution and $\usup$ is a supersolution. We proceed to show that $\usub$ is a subsolution since the proof for $\usup$ is similar.

Given a smooth test function $\phi$, let $x_0$ be a strict global maximum of $\usub$ with $\phi(x_0) = \usub(x_0)$. By Lemma \eqref{lemma:stabilitymaxima} below, we can find sequences with
\[
\begin{cases}
h_n \to 0\\
y_n \to x_0\\
u^{h_n}(y_n) \to \usub(x_0)
\end{cases}\]
where $y_n$ is a global maximizer of $u^{h_n}-\phi$.

Define
\bq\label{epsilondef}
\varepsilon_n = u^{h_n}(y_n)-\phi(y_n).
\eq
Then $\varepsilon_n \to \usub(x_0)-\phi(x_0) = 0$ and $u^{h_n}(x)-\phi(x) \leq u^{h_n}(y_n)-\phi(y_n) = \varepsilon_n$ for any $x \in \overline{\Omega}$. In particular,
\bq\label{aux1Appendix}
u^{h_n}(\cdot)-\phi(\cdot) \leq \varepsilon_n.
\eq
We know that
\[u(\cdot) \leq v(\cdot) \Rightarrow F_M^h[u] \geq F_M^h[v]\]
for any $u$ and $v$ bounded due to the monotonicity of the scheme (Definition \ref{monotoneAppendix}. Using now the definition \eqref{defnFilteredAppendix} of $F^h$ we get that for any $u$ and $v$ bounded
\[u(\cdot) \leq v(\cdot) \Rightarrow F^h[u] \geq F^h[v] - 2\sqrt{h},\]
since $\left| |\nabla u^h|^A - |\nabla u^h|^M\right| \leq \sqrt{h}$ and $F_M^h$ is monotone. Hence from \eqref{aux1Appendix} we conclude that
\bq\label{aux2Appendix}
F^{h_n}(x,s,u^{h_n}(\cdot)) \leq F^{h_n}(x,s,\phi(\cdot)+\varepsilon_n).
\eq
We then have
\begin{align*}
0	& = F^{h_n}[u^{h_n}](y_n) \text{ since $u^{h_n}$ is a solution}\\
	& = F^{h_n}(y_n, u^{h_n}(y_n),u^{h_n}(\cdot))\\
	& = F^{h_n}(y_n,\phi(y_n)+\varepsilon_n,u^{h_n}(\cdot)) \text{ by } \ref{epsilondef}\\
	& \geq F^{h_n}(y_n,\phi(y_n)+\varepsilon_n,\phi(\cdot)+\varepsilon_n) - 2 \sqrt{h} \text{ by } \ref{aux2Appendix}\\
\end{align*}
Finally, taking the $\liminf$ we get
\begin{align*}
0	& \geq \liminf_{n\to \infty} \left\{F^{h_n}(y_n,\phi(y_n)+\varepsilon_n,\phi(\cdot)+\varepsilon_n)-2\sqrt{h_n}\right\}\\
	& \geq \liminf_{h_n\to 0,y\to x_0,\varepsilon\to 0} F^{h_n}(y,\phi(y)+\varepsilon,\phi(\cdot)+\varepsilon)\\
	& = F_*(x_0,\phi(x_0),\nabla \phi(x_0))\\
	& = F_*(x_0,\usub(x_0),\nabla \phi(x_0))
\end{align*}
which shows that $\usub$ is a subsolution.
\end{proof}

\begin{lemma}[stability of maxima]\label{lemma:stabilitymaxima}
Suppose the family $u^h$ is bounded uniformly in $h$. Define 
\[\usub(x) = \limsup_{h\to 0, y\to x} u^h(u) \in \USC(\overline{\Omega}).\]
Given a smooth function $\phi$, let $x_0$ be a strict global maximum of $\usub-\phi$ with $\usub(x_0)=\phi(x_0)$. Then there exists sequences
\[
\begin{cases}
h_n \to 0\\
y_n \to x_0\\
u^{h_n}(y_n) \to \usub(x_0)
\end{cases}\]
where $y_n$ is a global maximizer of $u^{h_n}-\phi$.
\end{lemma}

\begin{proof}
From the definition of $\limsup$, there are sequences such that
\[
\begin{cases}
h_n \to 0,\\
z_n \to x_0,\\
u^{h_n}(z_n) \to \usub(x_0).
\end{cases}\]

Let $y_n \in \overline{\Omega}$ be the global maximizers of $u^{h_n}(\cdot)-\phi(\cdot)$. Then we have
\[u^{h_n}(y_n)-\phi(y_n) \geq u^{h_n}(z_n)-\phi(z_n) \to \usub(x_0)-\phi(x_0) = 0.\]
In addition, for any $\delta>0$ and large enough $n$,
\[u^{h_n}(y_n)-\phi(y_n) \leq \usub(y_n)-\phi(y_n)+\delta \leq \usub(x_0)-\phi(x_0)+\delta = \delta\]
where we used the fact that $x_0$ is a global maximum of $\usub-\phi$ with $\usub(x_0) = \phi(x_0)$. Thus we conclude that
\[u^{h_n}(y_n)-\phi(y_n)\to 0.\]

Now, we show by contradiction that $y_n\to x_0$. Suppose not. Then, by passing to a subsequence if needed there is an $R > 0$ such that $|y_n-x_0| > R$. Moreover, since $\usub-\phi$ has a strict, global and unique maximum at $x_0$ with value zero, there is a $K>0$ such that
\[\usub(y)-\phi(y) < -K\]
whenever $|y-x_0|>R$. For $n$ large enough we have
\[u^{h_n}(y_n) \leq \usub(y_n)+\frac{K}{2}\]
and so
\[u^{h_n}(y_n)-\phi(y_n) \leq \usub(y_n)-\phi(y_n) + \frac{K}{2} < -K+\frac{K}{2} = -\frac{K}{2}\]
which contradicts the fact that $u^{h_n}(y_n)-\phi(y_n) \to 0$. We then conclude that $y_n \to x_0$.

Finally we see that
\begin{align*}
|u^{h_n}(y_n)-\usub(x_0)|	& = |u^{h_n}(y_n)-\phi(x_0)|\\
							& \leq |u^{h_n}(y_n)-\phi(y_n)|+|\phi(y_n)-\phi(x_0)|\\
							& \to 0
\end{align*}
and therefore $u^{h_n}(y_n)\to \usub(x_0)$ as desired.

\end{proof}

\bibliographystyle{alpha}
\bibliography{../biblio/FilteredSchemes}

\end{document}